\documentclass[11pt,a4paper]{amsart}

\usepackage{a4wide}

\usepackage{amsfonts,amsthm,amssymb,mathabx}
\usepackage[pdftex]{graphicx}
\usepackage{graphics}
\usepackage[matrix,arrow]{xy}
\usepackage[active]{srcltx}

 \usepackage[colorlinks=true]{hyperref}
 \hypersetup{
     colorlinks   = true,
     citecolor    = blue
}

\newtheorem{thm}{Theorem}
\numberwithin{thm}{section}
\numberwithin{equation}{section}
\newtheorem{theorem}[thm]{Theorem}
\newtheorem*{theorem*}{Theorem}

\newtheorem{corollary}[thm]{Corollary}
\newtheorem*{corollary*}{Corollary}
\newtheorem{lemma}[thm]{Lemma}
\newtheorem{sublemma}[thm]{Sublemma}

\newtheorem{proposition}[thm]{Proposition}

\newtheorem*{conjecture*}{Conjecture}

\newtheorem*{question*}{Question}
\newtheorem{definition}[thm]{Definition}
\newtheorem*{definition*}{Definition}

\newtheorem*{definitions*}{Definitions}

\newtheorem*{rem*}{Remark}
\theoremstyle{remark}
\newtheorem{remark}[thm]{Remark}

\newtheorem*{remark*}{Remark}

\newtheorem*{remarks*}{Remarks}
\newtheorem*{example*}{Example}

\newtheorem*{examples*}{Examples}

\newcommand{\R}{\mathbb{R}}
\newcommand{\Z}{\mathbb{Z}}
\newcommand{\Q}{\mathbb{Q}}

\newcommand{\N}{\mathbb{N}}

\def\CP{{\mathbb C}{\mathbb P}}

\newcommand{\Id}{\mathit{Id}}

\newcommand{\ep}{\epsilon}
\newcommand{\ga}{\gamma}
\newcommand{\Ga}{\Gamma}

\def\lg{\langle}
\def\rg{\rangle}

\def\cz{{\mu_\text{CZ}}}

\def\rs{{\mu_\text{RS}}}
\def\maslov{{\mu_\text{Maslov}}}
\def\czl{{\mu^-_\text{CZ}}}
\def\czu{{\mu^+_\text{CZ}}}

\def\vr{\varphi}
\def\om{\omega}

\def\pr{\prime}
\def\P{{\mathcal P}}

\def\lo{{\mu^-_\text{CZ}}}
\def\mper{{{\mathfrak p}}}
\def\mperj{{{\mathfrak p}_j}}
\def\Pj{{P_{\gamma_j}}}
\def\Bott{{\mathfrak B}}
\def\q{{\mathfrak q}}

\newcommand{\Pp}{\mathcal{P}}

\newcommand{\tgamma}{\tilde{\gamma}}

\begin{document}
\title{Multiplicity of periodic orbits for dynamically convex contact forms}

\author[M.~Abreu]{Miguel Abreu}
\address{Center for Mathematical Analysis, Geometry and Dynamical Systems,
Instituto Superior T\'ecnico, Universidade de Lisboa, 
Av. Rovisco Pais, 1049-001 Lisboa, Portugal}
\email{mabreu@math.tecnico.ulisboa.pt}
 
\author[L.~Macarini]{Leonardo Macarini}
\address{Universidade Federal do Rio de Janeiro, Instituto de Matem\'atica,
Cidade Universit\'aria, CEP 21941-909 - Rio de Janeiro - Brazil}
\email{leomacarini@gmail.com}

\thanks{MA was partially funded by FCT/Portugal through UID/MAT/04459/2013 and 
project EXCL/MAT-GEO/0222/2012, and by CNPq/Brazil through a visiting grant. 
LM was partially supported by CNPq/Brazil and by FCT/Portugal through a visiting grant.
The present work is part of the authors activities within BREUDS, a research partnership 
between European and Brazilian research groups in dynamical systems,
supported by an FP7 International Research Staff Exchange Scheme (IRSES) 
grant of the European Union.}

\dedicatory{This paper is dedicated to Professor Paul H. Rabinowitz.}

\begin{abstract}
We give a sharp lower bound for the number of geometrically distinct contractible periodic orbits of dynamically convex Reeb flows on prequantizations of symplectic manifolds that are not aspherical.  Several consequences of this result are obtained, like a new proof that every bumpy Finsler metric on $S^n$ carries at least two prime closed geodesics, multiplicity of elliptic and non-hyperbolic periodic orbits for dynamically convex contact forms with finitely many geometrically distinct contractible closed orbits and precise estimates of the number of even periodic orbits of perfect contact forms. We also slightly relax the hypothesis of dynamical convexity. A fundamental ingredient in our proofs is the common index jump theorem due to Y. Long and C. Zhu.
\end{abstract}

\maketitle

\section{Introduction}

The problem of the existence of periodic orbits on energy levels of Hamiltonian flows is a long-standing question in Hamiltonian Dynamics. This problem traces back to Poincar\'e and had a strong revival after the pioneering work of Rabinowitz \cite{Rab1,Rab2,Rab3} who proved the existence of at least one periodic orbit on every starshaped energy hypersurface in $\R^{2n}$. 

Rabinowitz's work motivated the introduction of the Weinstein conjecture \cite{Wei} which asks the existence of a periodic orbit for every Reeb flow on a closed contact manifold and is one of the most fundamental driving questions in Symplectic Topology. Although this conjecture is open in general, several partial positive results have been achieved, in particular Taubes' celebrated proof in dimension three \cite{Tau}. 

A natural problem is what is beyond the Weinstein conjecture, namely, what one can say about the number of geometrically distinct periodic orbits of Reeb flows besides the existence of just one periodic orbit. As pointed out in \cite{GG14}, two complementary classes of contact manifolds can be considered in this problem: those for which the rank of the contact or symplectic homology groups grows as a function of the degree and those that do not have such homological growth. (Naturally, we are considering here contact manifolds for which some sort of nice symplectic or contact homology can be defined.)

The first group was considered in \cite{HM, McL} where the existence of infinitely many simple closed orbits for every Reeb flow on these manifolds is proved.  The main point is that the contribution to the homology of the iterations of an isolated periodic orbit is uniformly bounded with respect to the iteration. Therefore, the growth of the homology forces the existence of infinitely many simple orbits. These results were inspired by a classical result due to Gromoll and Meyer \cite{GM} which establishes the existence of infinitely many prime closed geodesics on a closed Riemannian manifold whose rank of the homology of the underlying loop space grows with respect to the degree.

The second group is more involved since, without homological growth, we do have examples of Reeb flows with finitely many simple periodic orbits. In this situation, the multiplicity of periodic orbits resembles the Hamiltonian Conley conjecture (HCC) for Hamiltonian diffeomorphisms on closed symplectic manifolds where there is also no (Floer) homological growth. The HCC states that every Hamiltonian diffeomorphism carries infinitely many simple periodic points \emph{whenever the symplectic manifold satisfies some conditions}. This conjecture was proved by Salamon and Zehnder \cite{SZ} for weakly non-degenerate Hamiltonian diffeomorphisms on aspherical symplectic manifolds and by Hingston \cite{Hin} for general Hamiltonian diffeomorphisms on the standard symplectic torus $T^{2n}$. These achievements culminate with Ginzburg's celebrated proof \cite{Gin} of the HCC for general Hamiltonian diffeomorphisms on aspherical symplectic manifolds (for generalizations of this result see \cite{GG14} and references therein). However, the HCC easily fails if we drop the assumption of asphericity: consider, for instance, an irrational rotation on $S^2$ which has only the poles as periodic points.

Similarly, the so called contact Conley conjecture (CCC) establishes that every Reeb flow (possibly meeting some extra natural conditions) on a contact manifold has infinitely many simple closed orbits whenever the contact manifold meets some conditions. With this conjecture in mind, our second group of contact manifolds naturally splits in two subgroups: the one given by contact manifolds that satisfy the CCC and the other for which the CCC fails.

In order to study these two subgroups, let us consider a prequantization $M^{2n+1}$ of a closed symplectic manifold $(B,\om)$ with contact structure $\xi$ satisfying $c_1(\xi)=0$. This can be considered as a prototypical example of a contact manifold with no homological growth. As a matter of fact, its contact homology is given by copies of the singular homology of the basis with a shift in the degree; see \cite{Bo}.

Under some mild extra assumptions, it was proved in \cite{GGM} that the Reeb flow of an index-admissible contact form on $M$ carries infinitely many simple closed orbits whenever $\om$ vanishes on $\pi_2(B)$. (Recall that a non-degenerate contact form $\alpha$ is \emph{index--admissible} if its Reeb flow has no contractible closed orbit with Conley--Zehnder index $2-n$ or $2-n\pm 1$ (this is the natural condition to have a well defined cylindrical contact homology).  In general, $\alpha$ is index--admissible when there exists a sequence of non-degenerate index--admissible forms $C^1$-converging to $\alpha$.) This furnishes a partial positive answer to the CCC for these prequantizations, showing that the first subgroup is not empty.

On the other hand, it is easy to see that the CCC fails when we drop the assumption that $\om$ is aspherical, as the example of the irrational ellipsoid shows. More generally, every prequantization of a closed symplectic manifold admitting a Hamiltonian circle action with finitely many fixed points does not satisfy the CCC and all these symplectic manifolds are necessarily not aspherical.

The purpose of this paper is to address the multiplicity problem of closed orbits on prequantizations when the basis is \emph{not aspherical} and therefore the CCC can fail. In order to attack this problem, we will consider \emph{dynamically convex} contact forms as defined in \cite{AM} (see Definition \ref{def:dynconvex}). This definition generalizes the classical notion of convexity for spheres in $\R^{2n}$ to general contact manifolds in a homological fashion. Our main result is Theorem \ref{MainThm} which establishes a sharp lower bound for the number of geometrically distinct contractible periodic orbits of dynamically convex contact forms on suitable prequantizations. Several consequences of this result and its proof are obtained, like a new proof that every bumpy Finsler metric on $S^n$ carries at least two prime closed geodesics (Corollary \ref{cor:2orbitsFinsler}), existence and multiplicity of elliptic and non-hyperbolic periodic orbits for dynamically convex contact forms with finitely many geometrically distinct contractible closed orbits (Theorems \ref{thm:elliptic} and \ref{thm:non-hyperbolic}) and precise estimates of the number of even periodic orbits for perfect contact forms (Theorem \ref{thm:perfect}). Finally, we show in Theorem \ref{thm:weak convexity} that the hypothesis of dynamical convexity in Theorem \ref{MainThm} can be slightly relaxed.

\vskip .2cm
\noindent {\bf Organization of the paper.} The results are stated in Section \ref{sec:results}. In Section \ref{sec:CH&index} we review the basic background on index theory and contact homology used in this work. Section \ref{sec:CIJT} presents the common index jump theorem due to Long and Zhu which is a crucial tool in the proof of our main theorem, established in Section \ref{sec:proof mainthm}. Section \ref{sec:proof elliptic&hyperbolic} gives the proofs of Theorems \ref{thm:elliptic} and \ref{thm:non-hyperbolic}. Finally, Theorem \ref{thm:perfect} is proved in Section \ref{sec:proof perfect}.

\subsection*{Acknowledgements} 
We thank IMPA and IST for the warm hospitality during the preparation of this work. We are grateful to Viktor Ginzburg, Basak Gurel and Jean Gutt for useful conversations regarding this paper. Part of these results were presented by the first author at the Workshop on Conservative Dynamics and Symplectic Geometry, IMPA, Rio de Janeiro, Brazil, August 3--7, 2015 and by the second author at the Contact and Symplectic Topology Session of the AMS-EMS-SPM Meeting, Porto, Portugal, June 10--13, 2015. They thank the organizers for the opportunity to participate in such wonderful events.

\section{Statement of the results}
\label{sec:results}

\subsection{Main result}

Let $(M,\xi)$ be a contact manifold such that its first Chern class $c_1(\xi)$ vanishes as an element of $H^2(M,\R)$. Assume that $\xi$ supports an index-admissible contact form and consider the cylindrical contact homology $HC_*(\xi)$ with rational coefficients and graded by the Conley-Zehnder index. Cylindrical contact homology has a natural filtration in terms of the free homotopy classes of the periodic orbits and we will denote by $HC^a_*(\xi)$ the homology generated by closed orbits with free homotopy class $a$. Given a periodic orbit $\ga$ we denote by $\czl(\ga)$ (resp. $\czu(\ga)$) a lower (resp. upper) semicontinuous extension of the Conley-Zehnder index to degenerate periodic orbits as defined in \cite{AM}; see Section \ref{sec:CH&index}. The following definition was introduced in \cite{AM} and generalizes the notion of dynamical convexity for tight contact forms on $S^3$ introduced in \cite{HWZ}. For a better explanation and motivation of this definition, we refer to \cite{AM}. In this work, we will consider a slightly weaker definition in the non-degenerate case. Recall that a non-degenerate periodic orbit is \emph{good} if its index has the same parity of the index of the underlying simple closed orbit; see Section \ref{sec:CH}.

\begin{definition}
\label{def:dynconvex}
Let $k_- = \inf\{k \in \Z; HC_k(\xi) \neq 0\}$ and $k_+ = \sup\{k \in \Z; HC_k(\xi) \neq 0\}$. A contact form $\alpha$ is positively (resp. negatively) dynamically convex if $k_-$ is an integer and $\czl(\gamma) \geq k_-$ (resp. $k_+$ is an integer and $\czu(\gamma) \leq k_+$) for every periodic orbit $\gamma$ of $R_\alpha$. Similarly, let $a$ be a free homotopy class in $M$, $k^a_- = \inf\{k \in \Z; HC^a_k(\xi) \neq 0\}$ and $k^a_+ = \sup\{k \in \Z; HC^a_k(\xi) \neq 0\}$. A contact form $\alpha$ is positively (resp. negatively) $a$-dynamically convex if $k^a_-$ is an integer and $\czl(\gamma) \geq k^a_-$ (resp. $k^a_+$ is an integer and $\czu(\gamma) \leq k^a_+$) for every periodic orbit $\gamma$ of $R_\alpha$ with free homotopy class $a$. If $\alpha$ is non-degenerate we consider only good closed orbits.
\end{definition}

In this work we will call a contact form $\alpha$ just \emph{dynamically convex} if it is either positively or negatively dynamically convex. There are tons of examples of dynamically convex contact forms, see \cite{AM} and the discussion below (in particular, as mentioned in the introduction, contact forms induced on convex spheres in $\R^{2n}$ are dynamically convex but there are several others).

The following theorem is our main result and addresses the multiplicity problem of closed orbits for dynamically convex contact forms on prequantizations of symplectic manifolds that are not aspherical. In the statement, $c_B := \inf\{k>0;\ \exists S \in \pi_2(B)\text{ such that }\lg c_1(TB),S \rg = k\}$ denotes the minimal Chern number of the symplectic manifold $B$ and $\chi(B)$ is its Euler characteristic.

\begin{theorem}
\label{MainThm}
Let $(M^{2n+1},\xi)$ be a prequantization of a closed symplectic manifold $(B,\om)$ such that $c_1(\xi)|_{\pi_2(M)}=0$, $\om|_{\pi_2(B)}\neq 0$ and $c_B\geq n$. Let $\alpha$ be a $0$-dynamically convex contact form supporting $\xi$. Then the following assertions hold:
\begin{itemize}
\item[(A)] Suppose that $\chi(B)\neq 0$. Then the Reeb flow of $\alpha$ carries at least two geometrically distinct contractible periodic orbits.
\item[(B)] Suppose, if $c_B=n$, that $H_{k}(B;\Q)=0$ for every odd $k$. If $\alpha$ is non-degenerate then its Reeb flow carries at least $r_B$ geometrically distinct contractible periodic orbits, where $r_B$ is the total rank of $H_\ast (B;\Q)$.
\end{itemize}
\end{theorem}

\begin{remark}
The hypothesis of dynamical convexity can be slightly relaxed; see Section \ref{sec:weak convexity}. The hypotheses on $\chi(B)$ and $H_\ast(B;\Q)$ are probably just technical and should be relaxed.
\end{remark}

\begin{remark}
The assumption $c_B\geq n$ can be replaced by the hypothesis that the Robbin-Salamon index of the smallest contractible multiple of the orbits of the circle action on $M$ is bigger than or equal to $2n$ or less than or equal to $-2n$; see Section \ref{sec:proof mainthm}. One can check that this hypothesis ensures that $M$ admits index-admissible contact forms. Consequently, its cylindrical contact homology is well defined. This hypothesis enters in the proof of the theorem in several ways and we do not know so far how to relax this.
\end{remark}

\begin{remark}
The condition on $c_1(\xi)$ is equivalent to the monotonicity of $B$ on $\pi_2(B)$, that is, $c_1(TB)|_{\pi_2(B)}=\lambda\omega|_{\pi_2(B)}$ for some $\lambda \in \R$. It is easy to see that if $\lambda>0$ (resp. $\lambda<0$) then $k^0_-$ is an integer and $k^0_+ = \infty$ (resp. $k^0_- = -\infty$ and $k^0_+$ is an integer). Thus, we have positive (resp. negative) $0$-dynamical convexity when $B$ is positive (resp. negative) monotone. Although there is no known example of a not aspherical positive monotone symplectic manifold with $c_B>n+1$, there are plenty of examples of negative monotone symplectic manifolds with arbitrarily high minimal Chern number; see \cite[pages 429-430]{MS}.
\end{remark}

\begin{remark}
This result is sharp in the non-degenerate case in any dimension as the examples of the irrational ellipsoid and the Katok-Ziller Finsler metrics show; see Remark \ref{rmk:sharp} below. These examples also show that the first statement of the theorem is sharp in dimension three, though so far no degenerate example with finitely many closed orbits is known.
\end{remark}

\begin{remark}
In dimension three $M$ has to be $S^3$ or a lens space and the theorem follows from the general result,  proved in \cite{CGH}, that every Reeb flow on a closed 3-manifold carries at least two simple closed orbits. 
\end{remark}

\begin{remark}
\label{rmk:transversality}
As it is well known, a rigorous construction of cylindrical contact homology has foundational transversality issues which we expect to be solved with the ongoing work on polyfolds of Hofer, Wysocki and Zehnder \cite{HWZ10,HWZ11,HWZ14}; see also the recent work of Pardon \cite{Par}. However, the proof of Assertion B works in a straightforward way for (the positive part of the) $S^1$-equivariant symplectic homology $SH^{+,S^1}(M,\xi)$ introduced by Bourgeois and Oancea \cite{BO} whenever the isomorphism \eqref{eq:CH} holds for $SH^{+,S^1}(M,\xi)$. This is the case, for instance, if $B$ admits a Morse function admitting only critical points with even index (see \cite{Gut}) and this is enough for the main applications in this work (indeed, one can check that this condition holds when $M$ is given by the standard contact sphere $S^{2n+1}$ or the unit sphere bundle $S^*S^{n+1}$; see \cite{AM}). The definition of equivariant symplectic homology usually requires a symplectic filling of $M$ but, as explained in \cite[Section 4.1.2]{BO}, this filling is not necessary if every contractible periodic orbit $\ga$ of $\alpha$ satisfies $\cz(\ga) + n - 2 > 1$. By dynamical convexity and our assumptions, this condition is fulfilled by $\alpha$ if $n>1$ (assuming that $\alpha$ is positively dynamically convex and that the lower bound for the index of the closed orbits holds for every contractible closed orbit, good or not); when $n=1$ the prequantization $M$ must be $S^3$ or a lens space and in this case, as already observed, the theorem follows from the general fact that every Reeb flow on a closed 3-manifold carries at least two closed orbits.

The proof of Assertion A should also work with $SH^{+,S^1}(M,\xi)$. The only necessary ingredient is a local version of equivariant symplectic homology similar to the one constructed in \cite{HM} for contact homology. In fact, we use only the Morse inequalities, resonance relations and the fact that the local homology of an SDM is concentrated in the highest degree; see Section \ref{sec:proof mainthm} for details. Curiously, in the proof of Assertion A we do not have to use \cite[Theorem 1.2]{GHHM} which establishes that, under some assumptions on $M$, the presence of a simple SDM forces the existence of infinitely many simple closed orbits.
\end{remark}

Examples of contact manifolds satisfying the hypotheses of Theorem \ref{MainThm} that will be particularly interesting for us are given by the standard contact sphere $S^{2n+1}$ and the sphere bundle $S^*S^{n+1}$ of $S^{n+1}$, with $n \geq 1$, endowed with the standard contact structure induced by the Liouville form on $T^*S^n$. Indeed, $S^{2n+1}$ is the prequantization of $\CP^n$ and $S^*S^{n+1}$ is the prequantization of $G^+_2(\R^{n+2})$, where $G^+_2(\R^{n+2})$ is the Grassmannian of oriented two-planes in $\R^{n+2}$ (see \cite{AM}). A computation shows that $c_{\CP^n}=n+1$, $c_{G^+_2(\R^{n+2})}\geq n$ (more precisely, $c_{G^+_2(\R^{3})}=2$ and $c_{G^+_2(\R^{n+2})} = n$ for $n>1$), $\dim H_\ast(\CP^n,\Q) = n+1$, $\dim H_\ast(G^+_2(\R^{n+2}),\Q) = 2(\lfloor \frac{n}{2}\rfloor+1)$ and $H_k(G^+_2(\R^{n+2}),\Q) = 0$ for every odd $k$, where, for a given $c \in \R$, $\lfloor c \rfloor := \max\{k \in \Z;\, k\leq c\}$. Thus, we obtain the following immediate corollaries.

\begin{corollary}
\label{cor:standard sphere}
Let $\alpha$ be a dynamically convex contact form on the standard contact sphere $S^{2n+1}$ with $n\geq 1$. Then the Reeb flow of $\alpha$ carries at least two simple closed orbits. Moreover, if $\alpha$ is non-degenerate then it carries at least $n+1$ simple periodic orbits.
\end{corollary}

The first statement generalizes a classical result due to Ekeland and Hofer \cite{EH} which establishes that every convex hypersurface in $\R^{2n}$, $n\geq 2$, has at least two geometrically distinct closed characteristics. The second statement was proved before by Gutt and Kang \cite{GK} and generalizes a result due to  Long and Zhu \cite{LZ}. It must be mentioned here that a crucial ingredient in our proofs, as well as in \cite{GK}, is the common index jump theorem due to Long and Zhu \cite{LZ}; see Section \ref{sec:CIJT}.

\begin{corollary}
\label{cor:sphere bundle of S^n}
Every $0$-dynamically convex contact form $\alpha$ on $S^*S^{n+1}$ endowed with the standard contact structure has at least two simple closed orbits. Moroever if $\alpha$ is non-degenerate then it has at least $2(\lfloor \frac{n}{2}\rfloor+1)$ simple periodic orbits.
\end{corollary}

\begin{remark}
The hypothesis of dynamical convexity in Corollaries \ref{cor:standard sphere} and \ref{cor:sphere bundle of S^n} can be slightly relaxed; see Section \ref{sec:weak convexity}.
\end{remark}

As an application, consider a Finsler metric $F$ on $S^n$. Denote by $K$ the flag curvature of $F$ and let $\lambda$ be its reversibility. As observed in \cite{AM}, the pinching condition $(\frac{\lambda}{\lambda+1})^2 < K \leq 1$ implies that the contact form corresponding to $F$ (whose Reeb flow is the geodesic flow) is $0$-dynamically convex. Therefore, we obtain the following result.

\begin{corollary}
\label{cor:geodesic flow}
Let $F$ be a Finsler metric on $S^{n+1}$, $n\geq 1$, satisfying the pinching condition $(\frac{\lambda}{\lambda+1})^2 < K \leq 1$. Then $F$ has at least two prime closed geodesics. Moreover, if $F$ is bumpy then it carries at least $2(\lfloor \frac{n}{2}\rfloor+1)$ prime closed geodesics.
\end{corollary}

The first assertion follows from a more general result due to Duan and Long \cite{DL2}. The second assertion was proved by Wang \cite{Wa1}.

\begin{remark}
The number $2(\lfloor \frac{n}{2}\rfloor+1)$ coincides with the number of prime periodic orbits of the examples due to Katok-Ziller \cite{Zil} of Finsler metrics on $S^{n+1}$ with finitely many prime closed geodesics (these examples have constant flag curvature equal to one). Motivated by these examples, Anosov conjectured in 1974 that every Finsler metric on $S^{n+1}$ has at least  $2(\lfloor \frac{n}{2}\rfloor+1)$ prime closed geodesics. The second assertion of the previous corollary, due to Wang, proves partially this conjecture.
\end{remark}

\begin{remark}
\label{rmk:sharp}
The irrational ellipsoid and the examples of Katok-Ziller mentioned in the previous remark show that Corollaries \ref{cor:standard sphere} and \ref{cor:sphere bundle of S^n} are sharp.
\end{remark}

Now, let us notice that a simple observation shows that Theorem \ref{MainThm} gives a lower bound for the number of simple periodic orbits of non-degenerate contact forms without the hypothesis of dynamical convexity. Indeed, let us recall that a contact form is called \emph{perfect} if it is non-degenerate and its differential in contact homology vanishes. (Here we can tacitly consider the contact homology filtered by some free homotopy class.) This means that every good periodic orbit is \emph{homologically necessary}. Clearly, from the very definition, every perfect contact form is dynamically convex.

A non-degenerate contact form is \emph{geometrically perfect} if the Conley-Zehnder index of every good periodic orbit has the same parity. (As before, we can tacitly consider only periodic orbits in a fixed free homotopy class.) Obviously, a geometrically perfect contact form is perfect. Note that a non-degenerate contact form carrying only one simple contractible closed orbit is geometrically perfect (considering the contact homology generated by contractible orbits). In particular, it is $0$-dynamically convex. Thus, arguing by contradiction, we obtain the following result.

\begin{corollary}
\label{cor:2orbits}
Let $(M^{2n+1},\xi)$ be a prequantization of a closed symplectic manifold $(B,\om)$ such that $c_1(\xi)|_{\pi_2(M)}=0$, $\om|_{\pi_2(B)}\neq 0$ and $c_B\geq n$. Moreover, if $c_B=n$ suppose that $\chi(B) \neq 0$. Then every non-degenerate contact form on $M$ has at least two geometrically distinct contractible closed orbits.
\end{corollary}

In particular, we conclude the following corollary proved before by Duan and Long \cite{DL1} and Rademacher \cite{Rad}:

\begin{corollary}
\label{cor:2orbitsFinsler}
Every bumpy Finsler metric on $S^{n}$ carries at least two prime closed geodesics.
\end{corollary}

\subsection{Multiplicity of elliptic and non-hyperbolic periodic orbits}

The proof of Theorem \ref{MainThm} has the following byproducts concerning the existence and multiplicity of elliptic and non-hyperbolic closed orbits when the contact form has finitely many geometrically distinct contractible closed orbits. Recall that a closed orbit is elliptic (resp. hyperbolic) if every eigenvalue of its linearized Poincar\'e map has modulus equal to (resp. different from) one.

\begin{theorem}
\label{thm:elliptic}
Let $(M^{2n+1},\xi)$ be a prequantization of a closed symplectic manifold $(B,\om)$ such that $c_1(\xi)|_{\pi_2(M)}=0$, $\om|_{\pi_2(B)}\neq 0$ and $c_B\geq n$. Let $\alpha$ be a $0$-dynamically convex contact form supporting $\xi$ with finitely many geometrically distinct contractible closed orbits. Then the Reeb flow of $\alpha$ carries at least one elliptic contractible periodic orbit.
\end{theorem}

\begin{theorem}
\label{thm:non-hyperbolic}
 Let $(M^{2n+1},\xi)$ be a prequantization of a closed symplectic manifold $(B,\om)$ such that $c_1(\xi)|_{\pi_2(M)}=0$, $\om|_{\pi_2(B)}\neq 0$ and $c_B\geq n$. Assume that $n$ is odd and that $H_{k}(B;\Q)=0$ for every odd $k$. Let $\alpha$ be a non-degenerate $0$-dynamically convex contact form supporting $\xi$ with finitely many geometrically distinct contractible closed orbits. Then the Reeb flow of $\alpha$ carries at least $r_B$ geometrically distinct non-hyperbolic contractible periodic orbits.
\end{theorem}

By the discussion above we get the following corollary which generalizes previous results due to Long and Zhu \cite{LZ} and Wang \cite{Wa0,Wa2}.

\begin{corollary}
\label{cor:non-hyperbolic}
Let $\alpha$ be a dynamically convex contact form on $S^{2n+1}$ or $S^*S^{n+1}$. Suppose that $\alpha$ carries finitely many simple closed orbits. Then it carries at least one elliptic closed orbit. Moreover, if $\alpha$ is non-degenerate and $n$ is odd then it has at least $n+1$ geometrically distinct non-hyperbolic closed orbits.
\end{corollary}

\begin{remark}
The hypothesis of dynamical convexity in Theorems \ref{thm:elliptic} and \ref{thm:non-hyperbolic} and in Corollariy \ref{cor:non-hyperbolic} can be slightly relaxed; see Section \ref{sec:weak convexity}.
\end{remark}

\subsection{Multiplicity of periodic orbits for perfect Reeb flows}

To the best of our knowledge, so far all the known examples of Reeb flows on prequantizations with finitely many simple closed orbits are dynamically convex and carry exactly $r_B$ closed orbits. Therefore, we can ask the following

\vskip .2cm
\noindent
{\bf Question:} Let $M$ be a prequantization  of a closed symplectic manifold $(B,\om)$ such that $\om|_{\pi_2(B)} \neq 0$. Is it true that a contact form on $M$ carries either $r_B$ or infinitely many simple closed orbits?
\vskip .2cm

It is unknown even for $S^3$ and in this case this is a problem raised by Hofer, Wysocki and Zehnder \cite{HWZ}. The so called \emph{contact Hofer-Zehnder conjecture} establishes that a Reeb flow with a homologically unnecessary periodic orbit carries infinitely many simple closed orbits \cite{Gur}. Clearly, a non-degenerate contact form that is not dynamically convex has a homologically unnecessary periodic orbit. Therefore, one can hope that a contact form that is not dynamically convex carries infinitely many simple closed orbits. (We stress the fact that such a result is far from being known.) Consequently, one can weaken the first question in the following way:

\vskip .2cm
\noindent
{\bf Question:} Let $M$ be a prequantization  of a closed symplectic manifold $(B,\om)$ such that $\om|_{\pi_2(B)} \neq 0$. Is it true that a dynamically convex contact form on $M$ carries either $r_B$ or infinitely many simple closed orbits?
\vskip .2cm

This is true for the tight $S^3$ due to the seminal work of Hofer, Wysocki and Zehnder \cite{HWZ}. It should be mentioned here that, to the best of our knowledge, in all the known examples of prequantizations $M^{2n+1}$ admitting contact forms with finitely many simple closed orbits the basis is not negative monotone and has minimal Chern number $c_B$ less than or equal to $n+1$. Thus, it is conceivable that if $B$ is negative monotone or $c_B$ is big enough then we actually have infinitely many orbits. Indeed, in view of the resemblance with the Hamiltonian Conley conjecture mentioned in the introduction, the HCC holds for negative monotone symplectic manifolds \cite{GG11} and it is conjectured that if $c_B$ is big enough (e.g. $c_B > 2n$) then the HCC holds for $B$ \cite{GG14}.

These two questions seem to be pretty hard and completely out of reach so far. A possibly more feasible question is the following. As mentioned before, a perfect contact form is dynamically convex. There is no known example of a perfect contact form with infinitely many simple closed orbits. Therefore we can ask the following, which is a sort of very weak version of the second question:

\vskip .2cm
\noindent
{\bf Question:} Let $M$ be a prequantization  of a closed symplectic manifold $(B,\om)$ such that $\om|_{\pi_2(B)} \neq 0$. Is it true that a perfect  contact form on $M$ carries exactly $r_B$ simple closed orbits?
\vskip .2cm
 
When $M=S^3$ this is true due to the work of Bourgeois, Cieliebak and Ekholm \cite{BCE}. The next result gives a partial positive answer to this question. It was proved for the standard contact sphere by Gutt and Kang \cite{GK}. Before we state this, let us recall that a simple periodic orbit $\ga$ is \emph{even} if $\cz(\ga)$ and $\cz(\ga^2)$ have the same parity, that is, the even iterates of $\ga$ are good orbits.

\begin{theorem}
\label{thm:perfect}
Let $(M^{2n+1},\xi)$ be a simply connected prequantization of a closed symplectic manifold $(B,\om)$ such that $c_1(\xi)|_{\pi_2(M)}=0$, $\om|_{\pi_2(B)}\neq 0$ and $c_B\geq n$. Moreover, if $c_B=n$ suppose that $H_{k}(B;\Q)=0$ for every odd $k$. Then the Reeb flow of a perfect contact form supporting $\xi$ carries exactly $r_B$ simple even periodic orbits.
\end{theorem}

\begin{remark}
It was proved by Gurel in \cite{Gur} that if $c_B>n$ then a perfect contact form carries at most $r_B$ simple even periodic orbits. Moreover, she showed that if $c_B=n$ then a perfect contact form has at most $r_B+2$ simple even periodic orbits. Thus, the previous theorem improves her bounds when $c_B=n$.
\end{remark}

It would be interesting to remove the word ``even" in the previous statement. This would give a positive answer to the previous question under the assumptions of Theorem \ref{thm:perfect}.

\subsection{The hypothesis of dynamical convexity}
\label{sec:weak convexity}

Although the hypothesis of dynamical convexity is important in the proofs of our results, it is probably just technical and should be relaxed (see the discussion in the beginning of the previous section). In fact, the proof of Theorem \ref{MainThm} shows that we can slightly relax this hypothesis and gives the following result. Assertion (B) was proved by Gutt and Kang \cite{GK} when $M$ is the standard contact sphere.

\begin{theorem}
\label{thm:weak convexity}
Let $(M^{2n+1},\xi)$ be a prequantization of a closed symplectic manifold $(B,\om)$ such that $c_1(\xi)|_{\pi_2(M)}=0$. Suppose that $\om|_{\pi_2(B)}\neq 0$, $c_B\geq n$ and that $H_{k}(B;\Q)=0$ for every odd $k$. Let $\alpha$ be a contact form supporting $\xi$ such that every contractible periodic orbit $\ga$ of $\alpha$ satisfies $\czl(\ga) \geq n$. Then the following assertions hold:
\begin{itemize}
\item[(A)] If $n\ne 2$ then the Reeb flow of $\alpha$ carries at least two geometrically distinct contractible periodic orbits. The same is true when $n=2$
provided that $M$ admits a strong symplectic filling $(W,\tilde\om)$ such that $\tilde\om|_{\pi_2(W)}=c_1(TW)|_{\pi_2(W)}=0$ and that 
$\pi_1(M)$ is torsion free.
\item[(B)] If $\alpha$ is non-degenerate then its Reeb flow carries at least $r_B$ geometrically distinct contractible periodic orbits.
\end{itemize}
\end{theorem}

\begin{remark}
The hypothesis on $M$ when $n=2$ in Assertion A is due to the use of \cite[Theorem 1.2]{GHHM}. It is necessary only when $c_B>n$.
\end{remark}

\begin{remark}
The hypothesis $\czl(\ga) \geq n$ can be replaced by the condition $\czu(\ga) \leq -n$.
\end{remark}

Theorem \ref{thm:weak convexity} follows directly from the proof of Theorem \ref{MainThm}; see Remarks \ref{rmk:weak convexity 1} and \ref{rmk:weak convexity 2}. In a similar fashion, the hypothesis of dynamical convexity in Theorems \ref{thm:elliptic} and \ref{thm:non-hyperbolic} can be also relaxed by the assumption that every contractible closed orbit $\ga$ satisfies $\czl(\ga) \geq n$.

\section{Basic background on Contact Homology and Index Theory}
\label{sec:CH&index}

In this section we will present the basic background used in this work concerning the index of periodic orbits and contact homology.

\subsection{The Conley-Zehnder index for paths of symplectic matrices}

Let $\P(2n)$ be the set of paths of symplectic matrices $\Gamma:[0,1]\to Sp(2n)$ such that $\Gamma(0)=\text{Id}$, endowed with the $C^1$-topology. Consider the subset $\P^*(2n) \subset \P(2n)$ given by the non-degenerate paths, that is, paths $\Ga \in \P(2n)$ satisfying the additional property that $\Gamma(1)$ does not have $1$ as an eigenvalue. Following \cite{SZ}, one can associate to $\Gamma \in \P^*(2n)$ its Conley-Zehnder index $\cz(\Gamma) \in \Z$ uniquely characterized by the following properties:
\begin{itemize}
 \item {\bf Homotopy:} If $\Gamma_s$ is a homotopy of arcs in $\P^*(2n)$ then $\cz(\Gamma_s)$ is constant.
 \item {\bf Loop:} If $\phi: \R/\Z \rightarrow Sp(2n)$ is a loop at the identity and $\Gamma \in \P^*(2n)$ then $\cz(\phi\Gamma) = \cz(\Gamma) + 2\maslov(\phi)$, where $\maslov(\phi)$ is the Maslov index of $\phi$.
  \item {\bf Signature:} If $A \in \R^{2n\times 2n}$ is a symmetric non-degenerate matrix with all eigenvalues of absolute value less than $2\pi$ and $\Gamma(t)=\exp(J_0At)$, where $J_0$ is the canonical complex structure in $\R^{2n}$, then $\cz(\Ga)=\frac{1}{2}\text{Sign}(A)$.
\end{itemize}

There are different extensions of $\cz: \P^*(2n) \to \Z$ to degenerate paths in the literature. In this work we will use the Robbin-Salamon index $\rs: \P(2n) \to \frac{1}{2}\Z$ defined in \cite{RS} and lower and upper semicontinuous extensions denoted by $\czl$ and $\czu$ respectively. The later are defined in the following way: given $\Ga \in \P(2n)$ we set
\begin{equation}
\label{lower_CZ}
\czl(\Ga) = \sup_{U} \inf\{\cz(\Ga^\pr) \mid \Ga^\pr \in U \cap \P^*(2n)\} 
\end{equation}
and
\begin{equation}
\label{upper_CZ}
\czu(\Ga) = \inf_{U} \sup\{\cz(\Ga^\pr) \mid \Ga^\pr \in U \cap \P^*(2n)\}, 
\end{equation}
where $U$ runs over the set of neighborhoods of $\Ga$. The index $\czl$ coincides with the one defined in \cite[Definition 6.1.10]{Lon02}. We have the relations
\begin{equation}
\label{eq:czl x czu 1}
\czl(\Ga^{-1}) = -\czu(\Ga)
\end{equation}
and
\begin{equation}
\label{eq:czl x czu 2}
\czu(\Ga) = \czl(\Ga) + \nu(\Ga)
\end{equation}
for every $\Ga \in \P(2n)$, where $\nu(\Ga)$ is the geometric multiplicity of the eigenvalue one of $\Ga(1)$. Indeed, the first equality follows immediately from the fact that if $\Ga \in \P^*(2n)$ then $\cz(\Ga^{-1})=-\cz(\Ga)$ and the second one is a consequence of \cite[Theorem 6.1.8]{Lon02}.

\subsection{The Conley-Zehnder index of periodic Reeb orbits}
\label{sec:index_orbits}

Let $\ga$ be a periodic orbit of the Reeb vector field $R_\alpha$ and let $\Phi_t: \xi(\ga(t)) \to \R^{2n}$ be a symplectic trivialization of $\xi$ over $\ga$. Using this trivialization, the linearized Reeb flow furnishes the symplectic path
\begin{equation}
\Gamma(t) = \Phi(\gamma(t)) \circ d\phi_\alpha^t(\gamma(0))|_\xi \circ \Phi^{-1}(\gamma(0)),
\end{equation}
where $\phi^t_\alpha$ is the Reeb flow of $\alpha$. In this way, we have the indexes $\rs(\ga;\Phi)$, $\czl(\ga;\Phi)$ and $\czu(\ga;\Phi)$ which coincide with $\cz(\ga;\Phi)$ if $\ga$ is non-degenerate, that is, if its linearized Poincar\'e map does not have one as eigenvalue. It turns out that the parities of the Conley-Zehnder indexes of the even/odd iterates of a periodic orbit are the same, that is,
\[
\mu_{\text{CZ}}^\pm(\ga^{2j};\Phi^{2j}) \equiv \mu_{\text{CZ}}^\pm(\ga^{2k};\Phi^{2k})\ \text{and}\  
\mu_{\text{CZ}}^\pm(\ga^{2j-1};\Phi^{2j-1}) \equiv \mu_{\text{CZ}}^\pm(\ga^{2k-1};\Phi^{2k-1})\ \text{mod 2}\ \forall\ j,k \in \N,
\]
see \cite{Lon02} and \eqref{eq:czl x czu 1}. The mean index of $\ga$ is defined as
\[
\Delta(\ga;\Phi) = \lim_{k\to\infty}\frac{1}{k} \czl(\ga^k;\Phi^k).
\]
This limit exists. In fact, following \cite{Lon99,Lon02}, one can associate to $\Gamma$ Bott's index function $\Bott: S^1 \to \Z$ which satisfies the following properties:
\begin{itemize}
\item[(a)] $\czl(\ga^k;\Phi^k) = \sum_{z \in S^1;\ z^k=1} \Bott(z)$ for every $k \in \N$ (Bott's formula).
\item[(b)] The discontinuity points of $\Bott$ are contained in $\sigma(P_\ga) \cap S^1$, where $P_\ga$ is the linearized Poincar\'e map of $\ga$ and $\sigma(P_\ga)$ is its spectrum.
\item[(c)] $\Bott(z)=\Bott(\bar z)$ for every $z \in S^1$.
\item[(d)] The \emph{splitting numbers} $S^\pm(z) := \lim_{\ep\to 0^+} \Bott(e^{\pm i\ep}z)-\Bott(z)$ depend only on $P_\ga$ and satisfy
\begin{itemize}
\item[(d1)] $0 \leq S^\pm(z) \leq \nu(z)$ for every $z \in \sigma(P_\ga) \cap S^1$, where $\nu(z)$ is the geometric multiplicity of $z$;
\item[(d2)] $S^\pm(z)=S^\mp(\bar z)$ for every $z \in S^1$.
\item[(d3)] If $P_\gamma$ is given by a symplectic sum $P_1 \oplus P_2$ then the corresponding splitting numbers satisfy $S^\pm(z) = S^\pm_1(z) + S^\pm_2(z)$ for every $z \in S^1$.
\item[(d4)] If $P_\ga$ is the identity then $S^\pm(1)=n$.
\end{itemize}
\end{itemize}
For a definition of $\Bott$ and a proof of these properties we refer to \cite{Lon99,Lon02}. It follows from Bott's formula that
\[
\lim_{k\to\infty}\frac{1}{k} \czl(\ga^k;\Phi^k) = \frac{1}{2\pi} \int_{S^1} \Bott(z) \,dz.
\]
This shows, in particular, that the mean index is well defined. The mean index is continuous with respect to the $C^2$-topology in the following sense: if $\alpha_j$ is a sequence of contact forms converging to $\alpha$ in the $C^2$-topology and $\ga_j$ is a sequence of periodic orbits of $\alpha_j$ converging to $\ga$ then $\Delta(\ga_j) \xrightarrow{j\to\infty} \Delta(\ga)$ \cite{SZ}. It is well known that the mean index satisfies the inequality
\begin{equation}
\label{eq:mean index}
|\cz(\tilde\ga;\Phi) - \Delta(\ga;\Phi)| \leq n
\end{equation}
for every closed orbit $\ga$ and non-degenerate perturbation $\tilde\ga$ of $\ga$. (To be precise, $\Phi$ defines in a natural way a trivialization of $\xi$ over $\tilde\ga$, that we also denote by $\Phi$, which is unique up to homotopy.) Moreover, if $\ga$ is weakly non-degenerate (that is, if $\sigma(P_\ga) \neq \{1\}$) then the inequality above is strict. By \cite[Theorem 6.1.8]{Lon02} we also have the relation
\begin{equation}
\label{eq:index and nullity}
|\cz(\tilde\ga;\Phi) - \czl(\ga;\Phi)| \leq \nu(\ga),
\end{equation}
where $\nu(\ga)$ is the nullity of $\ga$, that is, the geometric multiplicity of the eigenvalue one of $P_\ga$; see \eqref{eq:czl x czu 2}. 

If we choose another trivialization $\Psi_t: \xi(\vr(t)) \to \R^{2n}$ over $\ga$ then we have the relations
\[
\mu_{\text{CZ}}^\pm(\ga;\Psi) = \mu_{\text{CZ}}^\pm(\ga;\Phi) + 2\maslov(\Psi_t \circ \Phi_t^{-1}).
\]
In particular, the parity of the index does not depend on the choice of the trivialization.

If $\ga$ is contractible, there is a way to choose the trivialization $\Phi$ unique up to homotopy if $c_1(\xi)|_{\pi_2(M)}=0$. In fact, consider a capping disk of $\ga$, that is, a smooth map $\sigma: D^2 \to M$, where $D^2$ is the two-dimensional disk, such that $\sigma|_{\partial D^2} = \ga$. Choose a trivialization of $\sigma^*\xi$ and let $\Phi$ be its restriction to the boundary, which gives a trivialization of $\xi$ over $\ga$. Since $D^2$ is contractible, the homotopy class of $\Phi$ does not depend on the choice of the trivialization of $\sigma^*\xi$. Moreover, the condition $c_1(\xi)|_{\pi_2(M)}=0$ ensures that the homotopy class of $\Phi$ does not depend on the choice of $\sigma$ as well.

\subsection{Contact homology}
\label{sec:CH}

Cylindrical contact homology is an invariant of the contact structure introduced by Eliashberg, Givental and Hofer \cite{EGH} that we will now very briefly recall. Let $\alpha$ be a non-degenerate contact form supporting $\xi$, that is, such that every closed orbit is non-degenerate. A periodic orbit of $\alpha$ is called good if its index has the same parity of the index of the underlying simple closed orbit (as noticed before, the parity of the index does not depend on the choice of the trivialization of $\xi$). A periodic orbit that is not good is called bad. Consider the chain complex $CC_*(\alpha)$ given by the graded group with coefficients in $\Q$ generated by good periodic orbits of $R_\alpha$ graded by their Conley-Zehnder indexes (throughout this paper we are not using the standard convention where the grading of contact homology is given by the Conley-Zehnder index plus $n-2$). The differential is defined counting holomorphic cylinders asymptotic to periodic orbits in the symplectization of $M$. Clearly, the chain complex depends on $\alpha$ but it turns out that, under suitable transversality assumptions, the corresponding homology $HC_*(\alpha)$ is an invariant of the contact structure (see Remark \ref{rmk:transversality}).

There are two natural filtrations in cylindrical contact homology in terms of the action and the free homotopy classes of the periodic orbits. More precisely, let $A_\alpha(\ga) := \int_\gamma \alpha$ be the action of a periodic orbit $\ga$ and take a real number $T>0$. Consider the chain complex $CC^{a,<T}_*(\alpha)$ (resp. $CC^{a,\leq T}_*(\alpha)$) generated by good periodic orbits of $R_\alpha$ with free homotopy class $a$ and action less than (resp. less than or equal to) $T$. This is a subcomplex of $CC_*(\alpha)$ and its corresponding homology is denoted by $HC^{a,<T}_*(\alpha)$ (resp. $HC^{a,\leq T}_*(\alpha)$). Given $T^\pr>T$, the complex $CC^{a,\leq T}_*(\alpha)$ is a subcomplex of $CC^{a,<T^\pr}_*(\alpha)$ and therefore we can consider the homology of the quotient $CC^{a,<T^\pr}_*(\alpha)/CC^{a,\leq T}_*(\alpha)$ denoted by $HC^{a,(T,T^\pr)}_*(\alpha)$. The filtered contact homology does depend on $\alpha$. However, given another non-degenerate contact form $\alpha^\pr$ and $T,T^\pr \in \R$ we have that $HC^{a,(T,T^\pr)}_*(\alpha) \cong HC^{a,(T,T^\pr)}_*(\alpha^\pr)$ if there exists a family of contact forms $\alpha_s$, $s \in [0,1]$, such that $\alpha_0=\alpha$, $\alpha_1=\alpha^\pr$ and $T, T^\pr \notin \Sigma(\alpha_s)$ for every $s$, where $\Sigma(\alpha_s) := \{A_{\alpha_s}(\gamma); \gamma\text{ is a periodic orbit of }\alpha_s\}$ is the action spectrum of $\alpha_s$. This enables us to define the filtered contact homology $HC^{a,(T,T^\pr)}_*(\alpha)$ of a possibly degenerate contact form $\alpha$ as $HC^{a,(T,T^\pr)}_*(\alpha^\pr)$ whenever $T,T^\pr \notin \Sigma(\alpha)$, where $\alpha^\pr$ is a non-degenerate $C^\infty$-perturbation of $\alpha$.

Let $b^a_i = \dim HC^a_i(\xi)$ be the $i$-th Betti number and assume that there exist integers $i_-$ and $i_+$ such that $b^a_i$ is finite for every $i\leq i_-$ and $i\geq i_+$. Under this assumption the positive/negative mean Euler characteristic is defined as
\begin{equation}
\label{eq:def_MEC}
\chi^a_\pm(\xi) = \lim_{j\to\infty} \frac{1}{j} \sum_{i=|i_\pm|}^j (-1)^i b^a_{\pm i}
\end{equation}
provided that this limit exists.

\subsection{Local contact homology}
\label{sec:LCH}

An important tool throughout this paper is local contact homology introduced in \cite{HM}. These homology groups are building blocks for the global contact homology and its definition goes as follows. As in the global case, we will be very sketchy and we refer to \cite{HM} for details. Let $\alpha$ be any contact form representing $\xi$ and $\ga$ a (possibly degenerate) isolated closed orbit of $\alpha$. Consider a sufficiently small tubular neighborhood $U \cong B \times S^1$ of $\ga$, where $B \subset \R^{2n}$ is a ball, and write $\ga=\psi^k$, where $\psi$ is the underlying simple closed orbit. Let $\alpha^\pr$ be a non-degenerate $C^\infty$-perturbation of $\alpha$. The local contact homology $HC_*(\ga)$ is given by the chain complex generated by the good periodic orbits of $\alpha^\pr$ in $U$ with homotopy class $k[\psi]$, where $[\psi]$ is the homotopy class of $\psi$ in $U$. The differential counts holomorphic cylinders asymptotic to these orbits in the symplectization of $U$ with respect to a suitable almost complex structure $J$. It turns out that the corresponding homology does not depend on the choices of $U$ and $J$ and it is invariant by isolated deformations of $\ga$, see \cite{HM}.

Let us summarize here some properties of $HC_*(\ga)$ that will be useful in this work and whose proofs can be found in \cite{HM}. Suppose that every periodic orbit of $\alpha$ with free homotopy class $a$ is isolated. Given $T \in \Sigma(\alpha)$ there exists $\ep>0$ such that
\[
HC_*^{a,(T-\ep,T+\ep)}(\alpha) \cong \bigoplus_{\ga} HC_*(\ga),
\]
where the sum runs over all the periodic orbits $\ga$ of $\alpha$ with free homotopy class $a$ and action $T$. For a non-degenerate periodic orbit $\gamma$ we have that 
\begin{equation*}
HC_{*}(\gamma) = \begin{cases}
\Q & \text{if}\ *=\cz(\ga)\ \text{and}\ \ga\ \text{is good} \\
0 & \text{otherwise}.
\end{cases}
\end{equation*}
For a general closed orbit $\ga$, it follows immediately from \eqref{eq:mean index}, \eqref{eq:index and nullity} and the definition of local contact homology that 
\begin{equation}
\label{eq:support CH}
HC_k(\ga)=0\ \text{whenever}\ k \notin [\Delta(\ga)-n,\Delta(\ga)+n] \cap [\czl(\ga),\czl(\ga)+\nu(\ga)].
\end{equation}
The local Euler characteristic of $\ga$ is defined as
\[
\chi(\gamma) = \sum_{i \in \Z} (-1)^i\dim HC_i(\gamma).
\]
The sum above is finite. The local {\it mean} Euler characteristic of $\gamma$ is defined as
\[
\hat\chi(\gamma) = \lim_{m\to\infty} \frac{1}{m} \sum_{k=1}^m \chi(\gamma^k).
\]
The limit above exists and is rational. In fact, it is proved in \cite{GGo} that the sequence $k \mapsto \chi(\ga^k)$ is periodic. Thus, if we denote by $\mper$ its period we have that
\[
\hat\chi(\ga) = \frac{1}{\mper} \sum_{k=1}^\mper \chi(\ga^k).
\]
Suppose now that $\alpha$ has finitely many simple closed orbits $\ga_1,\dots,\ga_q$ with negative/positive mean index and free homotopy class $a$ (simple here means that $\ga_j$ is not a non-trivial covering of a periodic orbit $\psi$ with free homotopy class $a$). Then, as proved in \cite{GGo}, the corresponding limit in \eqref{eq:def_MEC} exists and it is proved in \cite{HM} that
\begin{equation}
\label{eq:resonance}
\sum_{j=1}^{q} \frac{\hat\chi(\gamma_j)}{\Delta (\gamma_j)} = \chi_\pm^a(\xi).
\end{equation}

\begin{remark}
Although the precise statement in \cite{HM} is not for a fixed free homotopy class, the proof goes in a straightforward way to show the resonance relation \eqref{eq:resonance}.
\end{remark}

\subsection{Strongly degenerate maximum}

We say that an isolated closed orbit $\ga$ is a strongly degenerate maximum (SDM) if $\Delta(\ga) \in 2\Z$ and $HC_{\Delta(\ga)+n}(\ga) \neq 0$. It follows from the strict inequality in \eqref{eq:mean index} for weakly non-degenerate periodic orbits that $\ga$ must be totally degenerate, that is, $\sigma(P_\ga) = \{1\}$. Although we will not use this in the proof of Theorem \ref{MainThm}, we should mention that the importance of an SDM comes from \cite[Theorem 2]{GHHM} which states that, under some assumptions on $M$, the presence of a simple SDM forces the existence of infinitely many simple closed orbits.

In this work we will use the following properties of an SDM whose proofs can be found in \cite{GHHM}. Firstly, that if $\ga$ is an SDM then its local contact homology is concentrated in degree $\Delta(\ga)+n$. More precisely, we have that $HC_{\Delta(\ga)+n}(\ga) \cong \Q$ and $HC_k(\ga)=0$ for every $k \neq \Delta(\ga)+n$. Secondly, that a totally degenerate closed orbit $\ga$ is an SDM if and only if $\ga^k$ is an SDM for some $k \in \N$.

Finally, an isolated closed orbit $\ga$ is a strongly degenerate minimum (SDMin) if $\Delta(\ga) \in 2\Z$ and $HC_{\Delta(\ga)-n}(\ga) \neq 0$. The SDMin has similar properties to those of an SDM. In particular, its local contact homology is concentrated in degree $\Delta(\ga)-n$ and a totally degenerate closed orbit $\ga$ is an SDMin if and only if $\ga^k$ is an SDMin for some $k \in \N$; see \cite{GHHM}.

\section{The common index jump theorem}
\label{sec:CIJT}

The common index jump theorem (CIJT) proved by Long and Zhu \cite{LZ} will be a crucial ingredient in this work. Roughly speaking, this theorem establishes that given a finite collection of periodic orbits $\ga_1,\dots,\ga_q$ with positive mean index there are arbitrarily large iterates of these orbits whose indexes jump a common interval of degrees whose size depends on a lower bound for the indexes of $\ga_1,\dots,\ga_q$. Moreover, the theorem provides precise estimates of the indexes and nullities of these iterates.

In what follows, we will use several definitions and notation from Section \ref{sec:index_orbits}. Before we state the theorem, let us recall that given a periodic orbit $\ga$ its elliptic height $e(\ga)$ is defined as the sum of the algebraic multiplicities of the eigenvalues of $P_\ga$ in the unit circle. In the statement, recall that given $c \in \R$ we define $\lfloor c \rfloor = \max\{k \in \Z;\, k \leq c\}$.

\begin{theorem}[Common index jump theorem \cite{LZ}]
Let $\ga_1,\dots,\ga_q$ be closed orbits with positive mean index. Define $J = \{1,\dots,q\}$ and let $\Pj$ be the linearized Poincar\'e map of $\ga_j$. Fix $N_0 \in \N$ and $\q \in \N$ such that $\q\theta/\pi \in \Z$ whenever $e^{i\theta}$ is an eigenvalue of $\Pj$ such that $\theta/\pi \in \Q$ for every $j \in J$. There exist $N = k N_0$ for some $k\in\N$ and $m_j \in \N$ for each $j\in J$ such that
\begin{equation}
\label{eq:CIJT1}
\nu(\gamma_j) = \nu(\gamma_j^{2m_j-1}) = \nu(\gamma_j^{2m_j+1}),
\end{equation}
\begin{equation}
\label{eq:CIJT2}
\lo (\gamma_j^{2m_j-1}) = 2N - \lo (\gamma_j) - 2S_j^+(1),
\end{equation}
\begin{equation}
\label{eq:CIJT3}
\lo (\gamma_j^{2m_j+1}) = 2N + \lo (\gamma_j),
\end{equation}
\begin{equation}
\label{eq:CIJT4}
\lo(\gamma_j^{2m_j}) \geq 2N - \frac{e(\ga_j)}{2} \geq 2N - n,
\end{equation}
and
\begin{equation}
\label{eq:CIJT5}
\lo(\gamma_j^{2m_j}) + \nu(\gamma_j^{2m_j}) \leq 2N + \frac{e(\ga_j)}{2} \leq 2N + n,
\end{equation}
for every $j \in J$, where $e(\ga_j)$ is the elliptic height of $\ga_j$ and $S_j^+(1)$ is the splitting number of the eigenvalue one of $\Pj$. Moreover, given an arbitrarily small $\epsilon > 0$ one can choose $m_j$ of the form
\begin{equation}
\label{eq:m}
m_j = \left(\left\lfloor\frac{N}{\q\Delta (\gamma_j)}\right\rfloor + \delta_j\right) \q
\end{equation}
such that
\begin{equation}
\label{eq:epsilon}
\left| \frac{N}{\q\Delta (\gamma_j)} - \left\lfloor\frac{N}{\q\Delta(\gamma_j)}\right\rfloor - \delta_j \right| < \epsilon\,,
\end{equation}
where $\delta_j\in\{0,1\}$.
\end{theorem}

\begin{remark}
\label{rmk:CIJT}
Although the statement of the CIJT in \cite{LZ} does not contain explicitly the assertions \eqref{eq:m} and \eqref{eq:epsilon}, these follow directly from the proof, see \cite[Theorem 3.8]{Wa1}. Moreover, it also follows from the proof of the CIJT that given $\delta>0$, $m_j$ can be chosen such that
\[
\min \left\{ \frac{m_j \theta_{i,j}}{\pi} - \left\lfloor \frac{m_j \theta_{i,j}}{\pi} \right\rfloor\,,\ 
1 - \left( \frac{m_j \theta_{i,j}}{\pi} - \left\lfloor \frac{m_j \theta_{i,j}}{\pi}\right\rfloor\right)\right\} < \delta\,.
\]
for every $j \in J$ and every eigenvalue $e^{\sqrt{-1}\theta_{i,j}}$, $\theta_{i,j} \in [0,2\pi)$, of the linearized Poincar\'e map of $\gamma_j$ with absolute value equal to $1$.
\end{remark}

\section{Proof of the Main Theorem}
\label{sec:proof mainthm}

Along the proof we will use several ideas from \cite{Wa1}. Firstly, notice that it is enough to consider the case that $\alpha$ is positively $0$-dynamically convex. Indeed, using \eqref{eq:czl x czu 1} and \eqref{eq:czl x czu 2} one can check that the same argument applies {\it mutatis mutandis} to the case that $\alpha$ is negatively $0$-dynamically convex; see Remark \ref{rmk:negative convexity}. Thus, we will consider only the positive mean Euler characteristic $\chi^0_+(\xi)$ defined in Section \ref{sec:CH} and omit the subscript $+$ in the notation.

Let $\beta$ be the contact form on $M$  whose Reeb vector field $R_\beta$ generates a free circle action with minimal period normalized equal to one. Let $\varphi$ be a simple closed orbit of $R_\beta$. Since $\omega|_{\pi_2(B)} \ne 0$, there exists a minimal $k_0 \in \N$ such that $\varphi^{k_0}$ is contractible. The cylindrical contact homology of $(M,\xi)$ for contractible orbits is given by
\begin{equation}
\label{eq:CH}
HC^0_\ast (\xi) \cong \oplus_{k\in\N} H_{\ast -kI +n} (B; \Q)\,,
\end{equation}
where $I:= \rs (\varphi^{k_0})$ is the Robbin-Salamon index of $\varphi^{k_0}$, see \cite{Bo}. Hence, the minimal degree $k^0_-$ with nonzero contact homology is given by $k^0_- = I - n$. It is easy to see that $I$ is an even number greater than or equal to $2c_B$, see \cite{AM}. Our hypotheses imply that
\[
I \geq 2n \quad\text{and}\quad k^0_- = I - n \geq n\,.
\]
In what follows, {\it a closed orbit means a contractible closed orbit and a simple closed orbit means the smallest contractible multiple of a simple (not necessarily contractible) closed orbit unless it is explicitly stated}.

Assume that $\alpha$ has finitely many simple closed orbits and let $\Pp = \{\ga_1,\dots,\ga_q\}$ be the set of simple orbits with positive mean index and $J=\{1,\dots,q\}$. Given $N_0$, $\q$ and $\ep>0$, we can apply the CIJT to $\ga_1,\dots,\ga_q$ and get natural numbers $(N,m_1,\dots,m_q)$ satisfying \eqref{eq:CIJT1}, \eqref{eq:CIJT2}, \eqref{eq:CIJT3}, \eqref{eq:CIJT4}, \eqref{eq:CIJT5}, \eqref{eq:m} and \eqref{eq:epsilon}. Theorem 2.2 from \cite{LZ} tells us that
\begin{equation}
\label{eq:indexest}
\nu(\gamma_j^{m}) - \frac{e(\gamma_j)}{2} \leq \lo(\gamma_j^{m+1}) - \lo(\gamma_j^{m}) - \lo(\gamma_j)\,,\ \forall\  m\in\N\,,\ 
j\in J\,.
\end{equation}
Under our hypotheses, this implies the following monotonicity property:
\begin{equation}
\label{eq:indexmon}
\lo(\gamma_j^{m+1}) - \lo(\gamma_j^m) \geq k^0_- - n \geq 0\,,\ \text{i.e.}\quad
\lo(\gamma_j^{m+1}) \geq \lo(\gamma_j^m) \,,\ \forall\  m\in\N\,,\ j\in J\,.
\end{equation}
The following lemma will be a keystone in the proof of our main theorem.

\begin{lemma}
\label{lemma:key}
$N_0$, $\q$ and $\ep$ can be chosen such that there exists $j \in J$ with $HC_{2N+n}(\ga_j^{2m_j}) \neq 0$.
\end{lemma}

\begin{proof}
Let us start with the following sublemma.

\begin{sublemma}
We have that
\[
\sum_{j=1}^q m_j \hat{\chi} (\gamma_j) = N \chi^0(\xi) = (-1)^n \frac{N\chi (B)}{I}\,.
\]
\end{sublemma}

\begin{proof}
Since $I$ is even, it is easy to see from \eqref{eq:CH} that 
\[
\chi^0(\xi) = (-1)^n \frac{\chi (B)}{I},
\]
which proves the second equality. To prove the first equality, recall from Section \ref{sec:LCH} that the sequence $k \mapsto \chi(\ga_j^k)$ is periodic and, denoting by $\mperj$ its period, we have that
\[
\hat\chi(\ga_j) = \frac{1}{\mperj} \sum_{k=1}^\mperj \chi(\ga_j^k).
\]
In the CIJT, choose $N_0$ to be a multiple of $I$, $\q$ to be a common multiple of $\mper_1,\dots,\mper_q$ and $\ep>0$ such that $\ep < 1/|\sum_{j=1}^q \q \hat\chi (\gamma_j)|$. Using the resonance relation \eqref{eq:resonance} and the fact that $m_j$ satisfies \eqref{eq:m}, we have that
\begin{eqnarray}
N \chi^0(\xi) & = & \sum_{j=1}^q \frac{N\hat\chi (\gamma_j)}{\Delta (\gamma_j)}   \nonumber \\
& = & \sum_{j=1}^q  \hat\chi (\gamma_j) \left( \left\lfloor \frac{N}{\q\Delta (\gamma_j)}\right\rfloor + \delta_j\right)\q
+ \sum_{j=1}^q  \hat\chi (\gamma_j) \underbrace{\left( \frac{N}{\q\Delta (\gamma_j)} - 
\left\lfloor \frac{N}{\q\Delta (\gamma_j)}\right\rfloor - \delta_j\right)\q}_{\rho_j} \nonumber \\
& = & \sum_{j=1}^q  m_j \hat\chi (\gamma_j) + \sum_{j=1}^q  \hat\chi (\gamma_j)\rho_j. \nonumber
\end{eqnarray}
But by our choices of $N_0$, $\q$ and $\epsilon$ we have that $N\chi^0(\xi) \in\Z$, $m_j \hat\chi(\gamma_j) \in \Z$ for every $j \in J$ and $|\sum_{j=1}^q \hat\chi (\gamma_j)\rho_j| < 1$. Hence,
\[
N \chi^0(\xi) = \sum_{j=1}^q m_j \hat{\chi} (\gamma_j) \,.
\]
\end{proof}

We will also need the following Morse inequalities. Recall that an isolated periodic orbit $\ga$ is called homologically visible if $HC_*(\ga)$ does not vanish.

\begin{proposition}
\label{prop:morseineq}
Let $\alpha$ be a contact form on $M$ with finitely many simple closed orbits. If $\alpha$ is non-degenerate, suppose that every good closed orbit $\ga$ satisfies $\cz(\ga)\geq n$; otherwise, assume that every periodic orbit $\ga$ satisfies $\czl(\ga)\geq n$. Then every homologically visible periodic orbit of $\alpha$ has positive mean index. Moreover, we have the Morse inequalities
\[
c_k - c_{k-1} + \cdots \pm c_n \geq b^0_k - b^0_{k-1} + \cdots \pm b^0_n
\]
for every $k\geq n$, where $c_i := \sum_{j=1}^q \sum_{m\geq 1} \dim HC_i(\ga_j^m)$ and $b^0_i := \dim HC^0_i(\xi)$.
\end{proposition}

\begin{proof}
Let us prove the first assertion of the proposition. Arguing by contradiction, suppose that there exists a closed orbit $\bar\ga$ such that $\Delta(\bar\ga) \leq 0$ and $HC_*(\bar\ga) \neq 0$. Note that $\bar\ga$ must be degenerate. As a matter of fact, if $\bar\ga$ is non-degenerate we would have a strict inequality in \eqref{eq:mean index} and hence $\cz(\bar\ga) < n$. This contradicts our hypothesis, since $\bar\ga$ is homologically visible and therefore good. Thus, we can assume that every closed orbit $\ga$ satisfies $\czl(\ga) \geq n$ which implies, in particular, that $\Delta(\bar\ga) = 0$. By the very definition of $\czl(\bar\ga)$, every periodic orbit in a non-degenerate perturbation of $\bar\ga$ has index bigger than or equal to $n$. On the other hand, by \eqref{eq:mean index} every periodic orbit in a non-degenerate perturbation of $\bar\ga$ has index less than or equal to $n$. Therefore, every non-degenerate periodic orbit that comes from a bifurcation of $\bar\ga$ must have index equal to $n$. In particular, since $\bar\ga$ is homologically visible, we conclude that $HC_n(\bar\ga) \neq 0$ which implies that $\bar\ga$ is an SDM and consequently $\bar\ga^m$ is an SDM for every $m \in \N$.

Denote by $\Sigma(\alpha)$ the action spectrum of $\alpha$, let $T \notin \Sigma(\alpha)$ and take a non-degenerate $C^\infty$-perturbation $\alpha^\pr$ of $\alpha$ such that $\alpha^\pr$ has no periodic orbit with index less than $n$ and action less than $T$. The existence of $\alpha^\pr$ is ensured by our hypotheses. It follows from the long exact sequence for the filtered contact homology of $\alpha^\pr$ and standard arguments that for a given $k \geq n$ we have
\begin{equation}
\label{eq:filtered_morseineq}
c^T_k - c^T_{k-1} + \cdots \pm c^T_n \geq b^{0,T}_k - b^{0,T}_{k-1} + \cdots \pm b^{0,T}_n
\end{equation}
where
\[
c^T_i := \sum_{\substack{\ga\ \text{contractible} \\ A(\ga) < T}} \dim HC_i(\ga)\ \ \text{and}\ \ b^{0,T}_i := \dim HC^{0,<T}_i(\alpha^\pr) = \dim HC^{0,<T}_i(\alpha)
\]
and we are using the fact that $\alpha^\pr$ has no periodic orbit with index less than $n$ and action less than $T$ (see \cite{HM} for details). In particular, we have
\[
c^T_{n+1} - c^T_n \geq b^{0,T}_{n+1} - b^{0,T}_n.
\]
Since there are finitely many simple closed orbits, we have that there exists $B>0$ such that $c^T_{n+1} < B$ for every $T$ (notice that only periodic orbits with positive mean index may contribute to $c^T_{n+1}$). By the existence of $\bar\ga$, we conclude that given $C>0$ there exists $T^\pr>0$ such that $c^T_n > C$ for every $T>T^\pr$. Thus, given $D>0$ there exists $T^\pr>0$ such that
\[
b^{0,T}_n > D
\]
for every $T>T^\pr$. But this is impossible and consequently the first assertion is proved. Indeed, if the last inequality holds then we could take suitable multiples $\alpha_1:= C_1\alpha$ and $\alpha_2:= C_2\alpha$ such that $\alpha_1 = f_1\beta$ and $\alpha_2 = f_2\beta$, where $\beta$ is the contact form on $M$ that generates the circle action and $f_i: M \to \R$, $i=1,2$, are smooth functions such that $f_1(x)<1$ and $f_2(x)>2$ for every $x \in M$. Using a sandwich argument (see, for instance, \cite{AM}) we would conclude that $HC_n^{0,T}(\beta)$ has arbitrarily high rank if $T$ is taken sufficiently big. But a computation (in fact, essentially the same computation used to prove the isomorphism \eqref{eq:CH}) shows that this is impossible.

To prove the last assertion of the proposition, notice that fixed $k$ there exists $T>0$ such that $c^T_i = c_i$ for every $i\leq k$, because every homologically visible periodic orbit has positive mean index. Moreover, it turns out that one can choose $T$ such that $b^{0,T}_i = b^0_i$ for every $i\leq k$; see \cite[Lemma 7.8]{HM}. (The statement of \cite[Lemma 7.8]{HM} is only for degrees bigger than $n$ (observe that the contact manifold in \cite{HM} has dimension $2n-1$ and the grading of contact homology is given by the Conley-Zehnder index plus $n-3$) but the proof can be readily adapted in our context for degrees  bigger than or equal to $n$ since every periodic orbit of $\alpha^\pr$ that comes from a bifurcation of a periodic orbit of $\alpha$ with zero mean index has to be bad.) Thus, it follows from \eqref{eq:filtered_morseineq} that
\[
c_k - c_{k-1} + \cdots \pm c_n \geq b^0_k - b^0_{k-1} + \cdots \pm b^0_n
\]
as desired.
\end{proof}

Now, arguing by contradiction, assume that the lemma is not true, that is, $HC_{2N+n}(\ga_j^{2m_j}) = 0$ for every $j \in J$. It follows from this hypothesis and \eqref{eq:CIJT5} that $HC_k (\gamma_j^{2m_j}) = 0$ for all $k>2N+n -1$ and $j \in J$. Thus,  we have that
\begin{equation}
\label{eq:keyprop1}
\sum_{k=n}^{2N+n-1} (-1)^k c_k = \sum_{m\geq 1} \sum_{j=1}^{q} \sum_{k=n}^{2N+n-1}
(-1)^k \dim HC_k (\gamma_j^m) = \sum_{j=1}^q \sum_{m=1}^{2m_j} \chi (\gamma_j^m)\,,
\end{equation}
because $HC_k (\gamma_j^m) = 0$ for all $m>2m_j$, $k\leq 2N+n-1$, and $HC_{2N+n} (\gamma_j^{2m_j}) = 0$ for every $j\in J$. Let $p_j=m_j/\mperj$ (observe that $p_j$ is an integer because $\q$ is a multiple of $\mperj$) and note that
\begin{equation}
\label{eq:keyprop2}
\sum_{m=1}^{2m_j} \chi(\ga_j^m) = \sum_{l=0}^{2p_j-1} \sum_{q=1}^{\mperj} \chi(\ga_j^{l\mperj+q}) = 2p_j\sum_{q=1}^{\mperj} \chi(\ga_j^{q})
= 2p_j\mperj\hat\chi(\ga_j) = 2m_j\hat\chi(\ga_j).
\end{equation}
Since $N$ is a multiple of $I$, write $N = sI$ for some $s\in\N$. By \eqref{eq:keyprop1}, \eqref{eq:keyprop2} and the sublemma, we have that
\begin{equation}
\label{eq:keyprop3}
\sum_{k=n}^{2N+n-1} (-1)^k c_k = (-1)^n 2s \chi (B) \,.
\end{equation}
On the other hand, by \eqref{eq:CH},
\begin{equation}
\label{eq:keyprop4}
\sum_{k=n}^{2N+n-1} (-1)^k b^0_k = (-1)^n 2s\chi (B)  + (-1)^{n+1}
\end{equation}
because the left hand side counts $(2s\chi(B) -1)$ up to a sign given by $(-1)^n$. Then, by
Proposition \ref{prop:morseineq}, we get
\begin{equation}
\label{eq:keyprop5}
-2s\chi(B) = c_{2N+n-1} -  \cdots - (-1)^n c_n \geq
b^0_{2N+n-1} - \cdots - (-1)^n b^0_n = -2s\chi (B) +1
\end{equation}
which is the desired contradiction.
\end{proof}

\begin{remark}
\label{rmk:elliptic}
It follows from \eqref{eq:CIJT5} and \eqref{eq:support CH} that the periodic orbit $\ga_j$ given by the previous lemma must be elliptic.
\end{remark}

Now, for the sake of clarity, we will split the proof of Theorem \ref{MainThm} in the degenerate and non-degenerate cases.

\subsection{Proof of Assertion A}
Arguing by contradiction, assume that $\Pp=\{\ga_1\}$. For a shorter notation, we will omit the subscript $1$ in $\ga_1$, ${\mathfrak p}_1$, $\delta_1$ and $m_1$. Since $N$ is a multiple of $I$, we have from \eqref{eq:CH} that
\begin{equation} \label{eq:HC}
HC^0_{2N+\ast} (\xi) \cong 
\begin{cases}
H_{n+\ast} (B;\Q) \,,\ \text{if $-n < \ast < n$}, \\
\Q \,,\ \text{if $\ast = -n,n$ and $I>2n$}, \\
\Q \oplus \Q \,,\ \text{if $\ast = -n,n$ and $I=2n$}.
\end{cases}
\end{equation}

\begin{lemma}
\label{lemma:SDM}
$N_0$ and $\q$ can be chosen so that $\gamma^{2m}$ is an SDM.
\end{lemma}
\begin{proof}
We have to prove that $N_0$ and $\q$ can be chosen so that $\Delta (\gamma^{2m}) \in 2\Z$ and
\[
HC_{\Delta(\gamma^{2m}) + n} (\gamma^{2m}) \ne 0.
\]
First, note that \eqref{eq:resonance} yields
\[
\Delta (\gamma) = (-1)^n \frac{\hat{\chi}(\gamma) I}{\chi (B)} \in\Q,
\]
where we are using the hypothesis that $\chi(B) \neq 0$ and the fact that $\hat\chi(\ga)$ is rational. Choose $\q$ to be a multiple of $\mper$, $N_0$ to be a multiple of $\q \hat{\chi}(\gamma) I$ and $\epsilon$ in \eqref{eq:epsilon} sufficiently small as in Lemma \ref{lemma:key}. We then have the following sequence of implications:
\[
\Delta (\gamma) = (-1)^n \frac{\hat{\chi}(\gamma) I}{\chi (B)} \Rightarrow
\frac{N}{\q\Delta (\gamma)} \in \N \Rightarrow \delta = 0 \Rightarrow 2m = \frac{2N}{\Delta (\gamma)}
\Rightarrow \Delta (\gamma^{2m}) = 2m \Delta (\gamma) = 2N\,. 
\]
Finally, it follows from Lemma \ref{lemma:key} that $HC_{2N+n} (\gamma^{2m}) \ne 0$.
\end{proof}

Now, suppose that $n=1$. In this case, $B$ is the sphere and $M$ is $S^3$ or a lens space. In particular, $k^0_-=3$. It follows from dynamical convexity and \eqref{eq:CIJT3} that $\ga^k$ with $k>2m$ cannot contribute to $HC^0_{2N-1}(\xi)$ and $HC^0_{2N+1}(\xi)$ (observe that, by Proposition \ref{prop:morseineq}, $c_i \geq b^0_i$ for every $i\geq n$). On the other hand, it follows from \eqref{eq:indexest} and dynamical convexity that 
\[
\czl(\gamma^{k}) + \nu (\gamma^{k}) \leq \czl(\gamma^{2m-1})  \leq 2N - 3 \,,\ \forall k \leq 2m-2.
\]
Thus, only $\ga^{2m-1}$ and $\ga^{2m}$ can contribute to $HC^0_{2N-1}(\xi)$ and $HC^0_{2N+1}(\xi)$. Moreover, since $\czl(\gamma^{2m-1})  \leq 2N - 3$, $\gamma^{2m-1}$ can contribute only to $HC^0_{2N-1}(\xi)$. But if there is such contribution then $\nu(\ga^{2m-1})$ has to be equal to two which implies, by \eqref{eq:CIJT1}, that $\nu(\ga)=2$ $\implies$ $P_\ga$ is the identity $\implies$ $S^+(1)=1$ $\implies$ $\czl(\gamma^{2m-1})  < 2N - 3$ (by \eqref{eq:CIJT2}) furnishing a contradiction. Hence, $\ga^{2m}$ has to contribute to both $HC^0_{2N-1}(\xi)$ and $HC^0_{2N+1}(\xi)$ which is impossible because the local contact homology of an SDM is concentrated in only one degree.

So suppose from now on that $n\geq2$. Using~(\ref{eq:HC}) we have that
\[
HC^0_{2N+n-2} (\xi) \cong H_{2n-2} (B;\Q) \cong H_2 (B;\Q) \ne 0\,,
\]
because $B$ is symplectic. Hence, some iterate of $\gamma$ must contribute to $HC^0_{2N+n-2} (\xi)$. The iterate $\gamma^{2m}$ cannot contribute to $HC^0_{2N+n-2} (\xi)$ because it is an SDM and $\Delta (\gamma^{2m}) + n = 2N +n$, hence it only contributes to $HC^0_{2N+n}(\xi)$. Also, we already know that $\gamma^{2m+1}, \gamma^{2m+2}, \ldots$, do not contribute to $HC^0_{2N+n-2} (\xi)$ due to \eqref{eq:CIJT3} and \eqref{eq:indexmon}. Regarding lower order iterates, note that it follows from dynamical convexity and \eqref{eq:indexest} that
\[
\lo (\gamma^{k}) + \nu (\gamma^{k}) \leq \lo (\gamma^{2m-1})  \leq 2N - n \,,\ \forall k \leq 2m-2
\]
which means that $\gamma^{2m-2}$ and all lower order iterates cannot contribute to $HC^0_{2N+n-2} (\xi)$. 

Hence, $\gamma^{2m-1}$ has to contribute to $HC^0_{2N+n-2} (\xi)$. For that to be possible, we must have $\lo (\gamma^{2m-1}) + \nu (\gamma^{2m-1}) \geq 2N+n-2$. By \eqref{eq:CIJT2} and dynamical convexity,
\[
\lo(\gamma^{2m-1}) \leq 2N - n - 2S^+(1).
\]
Now, we claim that if $n>2$ then $S^+(1)>0$. Indeed, suppose that $n>2$ and denote by $P$ the linearized Poincar\'e map of $\ga$. By the discussion above and \eqref{eq:CIJT1}, we have that $\nu(\ga^{2m-1}) = \nu(\ga) = \dim \ker(P-\Id) \geq 2n-2 > n$ since $n>2$. Thus, $\ker(P-\Id)$ must contain a symplectic plane $V$ because, otherwise, we would have that $\ker(P-\Id)$ is isotropic and $\dim \ker(P-\Id) > n$, a contradiction. Consequently, we can write $P= P|_V \oplus P|_{V^\om}$, where $V^\om$ denotes the symplectic orthogonal to $V$. But this implies that $S^+(1) = S_{P|_V}^+(1) + S_{P|_{V^\om}}^+(1) \geq S_{P|_V}^+(1) = 1$, where the last equality follows from the fact that $P|_V$ is the identity.

If $n=2$ and $\ga$ is not totally degenerate then we claim that we also have that $S^+(1)>0$. As a matter of fact, if $\ga$ is not totally degenerate then the algebraic multiplicity of the eigenvalue one is equal to its geometric multiplicity $\nu(\ga)=\nu(\ga^{2m-1})=2$ ($\nu(\ga^{2m-1}) \geq 2$ because $\czl(\ga^{2m-1})+\nu(\ga^{2m-1}) \geq 2N+n-2$). But this implies that $\ker(P-\Id)$ is symplectic. Indeed,  there is always a symplectic subspace $V$ invariant by $P$ whose complexification is the generalized eigenspace of the eigenvalue one and, when the algebraic and geometric multiplicities of the eigenvalue one coincide, this subspace is given by $\ker(P-\Id)$. Arguing as before, we conclude the claim.

\begin{remark}
There is another proof that $S^+(1)>0$ if $n>2$ or if $n=2$ and $\ga$ is not totally degenerate. In fact, Proposition 9.1.11 from \cite{Lon02} implies that
\[
0 \leq \nu(\ga) - S^+(1) \leq p,
\]
where $(p,p)$ is the Krein type of the eigenvalue one of $P$ (see Definition 1.3.6 in \cite{Lon02}). Arguing by contradiction, suppose that $S^+(1)=0$. By definition, $p \leq n$ which implies that $2n-2 \leq n$. If $n>2$, the last inequality furnishes a contradiction. If $n=2$, we conclude that $p=2$ but this implies that $\ga$ is totally degenerate because the algebraic multiplicity of the eigenvalue one is at least $2p$.
\end{remark}

Now, if $S^+(1)>0$ then $\lo(\ga^{2m-1}) \leq 2N - n - 2$ $\implies$ $\nu(\ga^{2m-1}) = 2n$ since $\lo (\gamma^{2m-1}) + \nu (\gamma^{2m-1}) \geq 2N+n-2$. But if $\nu(\ga) = \nu(\ga^{2m-1}) = 2n$ then the linearized Poincar\'e map of $\ga$ is the identity and consequently $S^+(1)=n$ $\implies$ $\lo(\ga^{2m-1}) \leq 2N - 3n$, contradicting the relation $\lo (\gamma^{2m-1}) + \nu (\gamma^{2m-1}) \geq 2N+n-2$ because $n \geq 2$. This finishes the proof of the first statement of Theorem \ref{MainThm} except in the case that $n=2$, $\ga$ is totally degenerate and $S^+(1)=0$.

In order to deal with this remaining case, note that if $\ga$ is totally degenerate and $S^+(1)=0$ then the associated Bott's index function is constant equal to $\lo(\ga)$, implying that $\Delta(\ga^j) = \lo(\ga^j)$ for every $j$. Since $\ga^{2m}$ is an SDM and $\ga$ is totally degenerate, we have that $\ga^j$ is an SDM for every $j$. Consequently, $\ga^j$ contributes only to $HC^0_{\Delta(\ga^j)+2}(\xi) = HC^0_{\lo(\ga^j)+2}(\xi)$. But $\lo(\ga^j)+2 > \lo(\ga) \geq k^0_-$ for every $j$ and therefore there is no contribution to $HC^0_{k^0_-}(\xi) \neq 0$. This contradiction finishes the proof of Assertion A in Theorem \ref{MainThm}.

\begin{remark}
\label{rmk:weak convexity 1}
The proof of Assertion A shows that if $\alpha$ contains only one simple periodic orbit $\ga$ such that $\czl(\ga)\geq n$ then $n=2$ and $\ga$ is an SDM (without assuming dynamical convexity of $\alpha$). If $M$ admits a strong symplectic filling $(W,\tilde\om)$ such that $\tilde\om|_{\pi_2(W)}=c_1(TW)|_{\pi_2(W)}=0$ and $\pi_1(M)$ is torsion free (which implies that $\ga$ cannot be a covering of a non-contractible closed orbit) then \cite[Theorem 1.2]{GHHM} establishes that the presence of $\ga$ implies the existence of infinitely many closed orbits with contractible projection to $B$. However, since $\pi_1(M)$ is torsion free and $\om|_{\pi_2(B)} \neq 0$, we have that the map $\pi_*: \pi_1(M) \to \pi_1(B)$ induced by the projection $\pi: M \to B$ is injective. Hence, the presence of $\ga$ actually implies the existence of infinitely many geometrically distinct contractible closed orbits. This proves Assertion A in Theorem \ref{thm:weak convexity}.
\end{remark}

\subsection{Proof of Assertion B}
Assume that $\alpha$ is non-degenerate and that it has finitely many simple closed orbits. Let $\Pp = \{\ga_1,\dots,\ga_q\}$ be the set of good simple orbits (notice that a simple periodic orbit can be a multiple of a non-contractible closed orbit). Since $\ga_j$ is non-degenerate and $\cz(\ga_j) \geq n$, we have that every $\ga_j$ has positive mean index. We will prove that $\#\Pp \geq \dim H_\ast (B, \Q)$. Firstly, we will prove the following easy first step towards this.

\begin{proposition}
\label{prop:step1}
If every good closed orbit $\ga$ satisfies $\czl(\ga) \geq n$ then
\[
\#\Pp \geq \dim H_\ast (B, \Q)  - 2
\]
and if every good closed orbit $\ga$ satisfies $\czl(\ga) > n$ then
\[
\#\Pp \geq \dim H_\ast (B, \Q)\,.
\]
\end{proposition}
\begin{proof}
Choose $N_0$ to be a multiple of $I$ in the CIJT. Then we have that
\[
\sum_{k = 2N -n+1}^{2N+n-1} \dim HC_k(\xi) \geq \dim H_\ast (B;\Q) -2\,,
\]
because this range of degrees in contact homology contains a copy of $H_1 (B;\Q)$, $H_2 (B;\Q)$, \dots , $H_{2n-1} (B;\Q)$. Since only $\gamma_j^{2m_j}$,
$j\in J$, can contribute to $HC^0_k(\xi)$ when $2N-n+1 \leq k \leq 2N+n-1$, we conclude that $\#\Pp \geq \dim H_\ast (B;\Q) -2$ (observe here that if a simple closed orbit is bad then all its iterates are bad and therefore cannot contribute to the contact homology).

If every contractible closed orbit $\ga$ satisfies $\czl(\ga) > n$ then \eqref{eq:CIJT4} and \eqref{eq:CIJT5} imply that the range of contact homology degrees $k$ where we can guarantee that only $\gamma_j^{2m_j}$, $j\in J$, contribute to $HC^0_j(\xi)$ is slightly larger: $2N-n \leq k \leq 2N+n$. This range contains a full copy of $H_\ast (B;\Q)$, which implies that $\#\Pp \geq \dim H_\ast (B, \Q)$.
\end{proof}

This proposition implies that we can, and will, assume from now on that $k^0_-=n$. Moreover, its proof gives us simple periodic orbits $\gamma_1, \ldots,\gamma_r\in\Pp$, with $r = \dim H_\ast (B; \Q) -2$, such that $2N-n+1 \leq \cz(\gamma_j^{2m_j}) \leq 2N + n -1$ for all $1\leq j \leq r$, where $N = sI$, for some $s\in\N$, and $m_j\in\N$, $1\leq j \leq r$, are given by the CIJT. All these $\gamma_j^{2m_j}$, $1\leq j \leq r$, contribute to $HC^0_\ast (\xi)$, which implies in particular that all even iterates of these $\gamma_j$'s are good.

\begin{remark}
\label{rmk:nonhyp1}
Suppose that $HC_k(B)=0$ for every odd $k$. Then the periodic orbits $\ga_1^{2m_1},\dots,\ga_r^{2m_{r}}$ contribute to $HC^0_{2N-n+2}(\xi),HC^0_{2N-n+4}(\xi),\dots,HC^0_{2N+n-2}(\xi)$. If $n$ is odd this implies that the Conley-Zehnder indexes of these orbits are odd. In particular, $\ga_1,\dots,\ga_{r}$  cannot be hyperbolic because every even iterate of a hyperbolic orbit has even index. If $k^0_->n$ we actually get $r_B$ simple non-hyperbolic closed orbits.
\end{remark}

The next proposition gives us one more periodic orbit.

\begin{proposition}
\label{prop:step2}
There exists a simple periodic orbit $\gamma_{r+1}\in\Pp$ such that $\gamma_{r+1}^{2m_{r+1}}$
is good and $\cz(\gamma_{r+1}^{2m_{r+1}}) = 2N+n$.
\end{proposition}

\begin{proof}
Recall that for a non-degenerate periodic orbit $\gamma$ we have that $HC_\ast (\gamma) = 0$ if $\ast\ne \cz (\gamma)$, while
\begin{equation*}
HC_{\cz(\ga)}(\gamma) \cong \begin{cases}
\Q & \text{if}\ \ga\ \text{is good} \\
0 & \text{otherwise}.
\end{cases}
\end{equation*}
The result now follows immediately from Lemma \ref{lemma:key}.
\end{proof}

\begin{remark}
\label{rmk:nonhyp2}
As noticed in Remark \ref{rmk:elliptic}, $\ga_{r+1}$ has to be elliptic.
\end{remark}

The final periodic orbit needed to finish the proof of Theorem~\ref{MainThm} is given by the following proposition.

\begin{proposition}
\label{prop:step3}
There exists a simple periodic orbit $\gamma_0 \in \Pp$ such that $\gamma_0 \ne \gamma_j$ for all $1\leq j \leq r+1$.
\end{proposition}

\begin{proof}
Since $HC^0_{n} (\xi) \ne 0$, there exists at least one good periodic orbit with index $n$. Let us consider two separate cases.

\vskip .3cm
\noindent {\bf Case 1.} There exists a unique good periodic orbit $\gamma$ such that $\cz(\gamma) = n$.
\vskip .2cm

Let $\ga_0$ be the underlying simple closed orbit of $\ga$. By dynamical convexity and \eqref{eq:indexmon}, $\cz(\gamma_0)=n$ $\implies$ $\ga_0=\ga$ (observe that $\ga_0$ is good because $\ga$ is good). Suppose that $\gamma_0 = \gamma_p$ for some $p\in\{1,\ldots,r+1\}$. Then $\cz(\gamma_j) > n$ for all $j\in\{1,\ldots,r+1\}\setminus\{p\}$ (notice that $\ga_1,\dots,\ga_{r+1}$ are good), and it follows from our assumptions and the CIJT that these $\gamma_j$'s cannot contribute to $HC^0_{2N-n} (\xi)$. Since $\cz (\gamma_0^{2m_0}) = \cz (\gamma_p^{2m_p}) \geq 2N-n+1$, only $\gamma_0^m$ with $m\leq 2m_0 -1$ can contribute to $HC^0_{2N-n} (\xi)$.

\begin{lemma} 
\[
\cz (\gamma_0^{2m_0 -2}) < \cz (\gamma_0^{2m_0 -1})\,.
\]
\end{lemma}

\begin{proof}
Since $\gamma_0$ is non-degenerate and $\cz(\ga_0) \geq n$, there are non-negative integers $a$ and $b$ such that
\begin{equation}\label{eq:2star}
\cz (\gamma_0^m) = m a + \sum_{i=1}^b 2 \left\lfloor \frac{m\theta_i}{2\pi}\right\rfloor + b \,,
\end{equation}
where each $\theta_i \in (0,2\pi)$ is the argument of an eigenvalue of the linearized Poincar\'e map of $\gamma_0$ with absolute value equal to $1$; see~\cite{Lon02}. 

If $a>0$ we are done. So suppose that $a=0$. Since $\cz (\gamma_0^2) > n$ (note that every even iterate of $\ga_0=\ga_p$ is good), there exists at least one $\theta_i$ such that $\theta_i \geq \pi$ and the non-degeneracy of $\gamma_0^2$ implies that, in fact, $\theta_i > \pi$. 

By Remark \ref{rmk:CIJT}, given an arbitrary $\delta > 0$ the number $m_0\in\N$ can be chosen such that
\[
\min \left\{ \frac{m_0 \theta_i}{\pi} - \left\lfloor \frac{m_0 \theta_i}{\pi} \right\rfloor\,,\ 
1 - \left( \frac{m_0 \theta_i}{\pi} - \left\lfloor \frac{m_0 \theta_i}{\pi}\right\rfloor\right)\right\} < \delta\,.
\]
We claim that if
\[
\delta < \min \left\{ \frac{\theta_i}{\pi} -1 , 1 - \frac{\theta_i}{2\pi} \right\}
\]
then there exists $c\in\Z$ such that
\begin{equation} \label{eq:c}
\frac{(2m_0 - 2)\theta_i}{2\pi} < c < \frac{(2m_0 - 1)\theta_i}{2\pi}\,,
\end{equation}
which finishes the proof (using~(\ref{eq:2star}) and $\theta_i > \pi$). To prove~(\ref{eq:c}) note that
\begin{eqnarray}
\theta_i > \pi & \Rightarrow & \frac{m_0 \theta_i}{\pi}   -\frac{\theta_i}{\pi}  <
\left\lfloor \frac{m_o \theta_i}{\pi}\right\rfloor \label{first}\\
\text{and}\quad \theta_i < 2\pi & \Rightarrow & \left\lfloor \frac{m_o \theta_i}{\pi}\right\rfloor - 1 < 
\frac{m_0 \theta_i}{\pi} - \frac{\theta_i}{2\pi}\,,  \label{second}
\end{eqnarray}
while the choice for $\delta$ implies that at least one of the following inequalities is true:
\begin{equation} \label{third}
\frac{m_0 \theta_i}{\pi} - \left\lfloor \frac{m_0 \theta_i}{\pi} \right\rfloor <  \frac{\theta_i}{\pi} -1
\quad\text{or}\quad
1 - \left( \frac{m_0 \theta_i}{\pi} - \left\lfloor \frac{m_0 \theta_i}{\pi}\right\rfloor\right) < 1 - \frac{\theta_i}{2\pi} \,.
\end{equation}
Combining~(\ref{first}) with the second inequality of~(\ref{third}), we get
\[
\frac{m_0\theta_i}{\pi} - \frac{\theta_i}{\pi} <  \left\lfloor \frac{m_0 \theta_i}{\pi} \right\rfloor <
\frac{m_0\theta_i}{\pi} - \frac{\theta_i}{2\pi} 
\]
which implies~(\ref{eq:c}) with $c =  \left\lfloor \frac{m_0 \theta_i}{\pi} \right\rfloor$. Combining~(\ref{second}) with the first inequality of~(\ref{third}), we get
\[
\frac{m_0\theta_i}{\pi} - \frac{\theta_i}{\pi} <  \left\lfloor \frac{m_0 \theta_i}{\pi} \right\rfloor - 1 <
\frac{m_0\theta_i}{\pi} - \frac{\theta_i}{2\pi} 
\]
which implies~(\ref{eq:c}) with $c =  \left\lfloor \frac{m_0 \theta_i}{\pi} \right\rfloor - 1$. This finishes the proof of the lemma.
\end{proof}

Thus, only $\gamma_0^{2m_0 -1}$ can contribute to $HC^0_{2N-n} (\xi)$ which would then have dimension at most one. But  by our choice of $N$ and \eqref{eq:CH},
\[
\dim HC_{2N-n} (\xi) = \dim H_0(B;\Q)\oplus H_{2n} (B;\Q) = 2 \,,
\]
which is a contradiction. Hence, in this Case $1$ we do have that $\gamma_0 \ne \gamma_j$ for all $1\leq j \leq r+1$.

\begin{remark}
\label{rmk:nonhyp3}
Using \eqref{eq:CIJT4} and the fact that $\cz(\ga_{0}^{2m_0})=2N-n$ we infer that $e(\ga_{0})=2n$, that is, $\ga_{0}$ is elliptic.
\end{remark}

\vskip .3cm
\noindent {\bf Case 2.} There are at least two distinct good periodic orbits $\gamma$ and  $\gamma'$ such that $\cz (\gamma) = \cz (\gamma') = n$.
\vskip .2cm

In this case, the fact that $\dim HC^0_{n} (\xi) = \dim H_0 (B;\Q) = 1$ implies that there exists a good periodic orbit $\tgamma$ such that $\cz (\tgamma) = n +1$. Notice that $\tgamma$ cannot be an iterate of any $\gamma_j$, $1\leq j\leq r+1$, because the index of all iterates of all these $\gamma_j$ have the same parity as $n$. Indeed, $\gamma_j^{2m_j}$ contributes to $HC^0_\ast (\xi)$ and $HC^0_k (\xi) = 0$ for every $k$ with parity different from $n$, since we are assuming that $H_k(B;\Q) = 0$ whenever $k$ is odd. Hence, in this Case 2 we can take $\gamma_0$ to be the simple periodic orbit underlying $\tgamma$.

\begin{remark}
\label{rmk:nonhyp4}
Suppose that $H_k(B)=0$ for every odd $k$ and that $n\geq 3$ is odd. We claim that the existence of $\tgamma$ implies the existence of a simple non-hyperbolic periodic orbit $\psi$ different from $\ga_1,\dots,\ga_{r+1}$ (which may coincide with $\ga_0$). Indeed, suppose that $\ga_0$ is hyperbolic (otherwise we are done) and let $m_0$ be the number associated to $\ga_0$ given by the CIJT applied to $\ga_0,\ga_1,\dots,\ga_{r+1}$. Since $\ga_0$ is hyperbolic, we have by \eqref{eq:CIJT4} and \eqref{eq:CIJT5} that $\cz(\ga_0^{2m_0}) = 2N$. Moreover, $\cz(\tgamma) = \cz(\ga_0^j) = j\cz(\ga_0)$ for some $j \in \N$. But $\cz(\tgamma) = n+1$ and $\cz(\ga_0) \geq n$ $\implies$ $j=1$ (note that $\ga_0$ is good because $\tgamma$ is good). Thus, $\cz(\ga_0)=n+1$ is even $\implies$ $\ga_0^{2m_0}$ is good. By our hypotheses, $HC^0_{2N}(\xi)=0$ and therefore we must have a simple closed orbit $\psi$ such that $\cz(\psi^k)=2N+1$ for some $k \in \N$. By the CIJT and the fact that $n\geq 3$, we can assume that $k$ is even and $\psi$ is necessarily distinct from $\ga_1,\dots,\ga_{r+1}$. Thus, $\psi$ is non-hyperbolic, as claimed.
\end{remark}

\begin{remark}
\label{rmk:weak convexity 2}
It is easy to see from the proof of Assertion B in Theorem \ref{MainThm} that one can slightly relax the hypothesis of dynamical convexity to prove Assertion B in Theorem \ref{thm:weak convexity}. As a matter of fact, Propositions \ref{prop:step1} and \ref{prop:step2} give us simple periodic orbits $\ga_1,\dots,\ga_{r+1}$ such that $\ga_i^{2m_i}$ is good and $\ga_i^{2m_i} \in \{2N-n+2,2N-n+4,\dots,2N+n\}$ for every $i \in \{1,\dots,r+1\}$. Moreover, $\#\Pp \geq \dim H_\ast (B, \Q)$ if every good closed orbit $\ga$ satisfies $\cz(\ga)> n$. Therefore, we can assume that there exists a good simple closed orbit $\bar\ga$ such that $\cz(\bar\ga)=n$. If $I=2n$ then Proposition \ref{prop:step3} gives us a new simple orbit. If $I>2n$ then, since $HC^0_n(\xi)=0$, the presence of $\bar\ga$ implies the existence of a good contractible closed orbit $\ga$ such that $\cz(\ga)=n+1$. Since the indexes of every iterate of $\ga_1,\dots,\ga_{r+1}$ have the same parity of $n$, $\ga$ must be geometrically distinct from $\ga_1,\dots,\ga_{r+1}$.
\end{remark}

\begin{remark}
\label{rmk:negative convexity}
The proof of Theorem \ref{MainThm} can be readily adapted to the case that $\alpha$ is negatively dynamically convex. In order to check this, the main points are the following. Let $\{\ga_1,\dots,\ga_q\}$ be a set of periodic orbits with negative mean index. Applying the CIJT to the inverse of the linearized Reeb flow along the periodic orbits and using \eqref{eq:czl x czu 1} we get natural numbers $N,m_1,\dots,m_j$ such that
\[
\czu(\gamma_j^{2m_j-1}) = -2N - \czu(\gamma_j) + 2S_j^+(1),
\]
\[
\czu(\gamma_j^{2m_j+1}) = -2N + \czu(\gamma_j),
\]
\[
\czu(\gamma_j^{2m_j}) \leq -2N + \frac{e(\ga_j)}{2} \leq -2N + n,
\]
and
\[
\czu(\gamma_j^{2m_j}) - \nu(\gamma_j^{2m_j}) \geq -2N - \frac{e(\ga_j)}{2} \geq -2N - n,
\]
for every $j \in \{1,\dots,q\}$, where $S_j^+(1)$ is the splitting number of the eigenvalue one of the inverse of the linearized Poincar\'e map of $\ga_j$. Using \eqref{eq:czl x czu 1} and \eqref{eq:indexest} we conclude that if $\czu(\ga) \leq -n$ then
\begin{align*}
\czu(\ga^{m+1}) & \leq \czu(\ga^m) - \nu(\ga^m) + \frac{e(\ga)}{2} - n \\
& \leq \czu(\ga^m) - \nu(\ga^m)
\end{align*}
for every $m \in \N$. Moreover, it follows from \eqref{eq:czl x czu 2} and \eqref{eq:support CH} that the local contact homology $HC_*(\ga)$ is supported in the range of degrees $[\czu(\ga)-\nu(\ga),\czu(\ga)]$. Finally, in the proof of Assertion A we conclude that, instead of an SDM, $\ga$ must be an SDMin. 
\end{remark}
\end{proof}

\section{Proofs of Theorems \ref{thm:elliptic} and \ref{thm:non-hyperbolic}}
\label{sec:proof elliptic&hyperbolic}

\subsection{Proof of Theorem \ref{thm:elliptic}}
Follows immediately from Lemma \ref{lemma:key} and Remark \ref{rmk:elliptic}.

\subsection{Proof of Theorem \ref{thm:non-hyperbolic}}
When $n=1$, $B=S^2$, $M$ is $S^3$ or a lens space and $k^0_-=3$. Then we have from Proposition \ref{prop:step1} that there are two geometrically distinct contractible closed orbits $\ga_1,\ga_2$ such that $\cz(\ga_1^{2m_1})=2N-1$ and $\cz(\ga_2^{2m_2})=2N+1$ which implies that these orbits are not hyperbolic. Therefore, we can assume that $n\geq 3$. The theorem then follows from the proof of Theorem \ref{MainThm} and Remarks \ref{rmk:nonhyp1}, \ref{rmk:nonhyp2}, \ref{rmk:nonhyp3} and \ref{rmk:nonhyp4}.

\section{Proof of Theorem \ref{thm:perfect}}
\label{sec:proof perfect}

Since $\alpha$ is perfect, every good periodic orbit of $\alpha$ has index bigger than or equal to $k_- \geq n$. Moreover, $\alpha$ carries precisely one good periodic orbit with index $k_-$. Indeed, if we have more than one good periodic orbit with index $k_-$ then we would have that the rank of $HC_{k_-}(\xi) \cong \Q$ is bigger than one, a contradiction.

Therefore, by the proof of  Assertion B in Theorem \ref{MainThm}, we conclude that $\alpha$ carries at least $r_B$ simple periodic orbits $\ga_1,\dots,\ga_{r_B}$ such that $\ga^{2m_1}_1,\dots,\ga^{2m_{r_B}}_{r_B}$ are good $\implies$ $\ga_1,\dots,\ga_{r_B}$ are even.

We claim that $\ga_1,\dots,\ga_{r_B}$ are the only simple even periodic orbits of $\alpha$. As a matter of fact, suppose that we have another simple even periodic orbit $\ga_{r_B+1}$. Since $\ga_{r_B+1}$ is good and $\alpha$ is non-degenerate, $\ga_{r_B+1}$ must have positive mean index which enables us to apply the CIJT to the orbits $\ga_1,\dots,\ga_{r_B+1}$. If $k_->n$ then the proof of Theorem \ref{MainThm} shows that the periodic orbits $\ga_1^{2m_1},\dots,\ga_{r_B}^{2m_{r_B}}$ generate $HC_{2N-n}(\xi) \oplus \dots \oplus HC_{2N+n}(\xi)$. Since $2N-n \leq \cz(\ga_{r_B+1}^{2m_{r_B+1}}) \leq 2N+n$, $\ga_{r_B+1}$ is even and $\alpha$ is perfect, we would conclude that the rank of $HC_{2N-n}(\xi) \oplus \dots \oplus HC_{2N+n}(\xi)$ is bigger than $r_B$, a contradiction. So suppose that $k_-=n$ and assume, without loss of generality, that $\ga_1$ has index equal to $k_-$. By the proof of Theorem \ref{MainThm}, the periodic orbits $\ga_1^{2m_1-1}$, $\ga_1^{2m_1}$, $\ga_1^{2m_1+1}$ and $\ga_2^{2m_2},\dots,\ga_{r_B}^{2m_{r_B}}$ generate $HC_{2N-n}(\xi) \oplus \dots \oplus HC_{2N+n}(\xi)$. Since $2N-n \leq \cz(\ga_{r_B+1}^{2m_{r_B+1}}) \leq 2N+n$, $\ga_{r_B+1}$ is even and $\alpha$ is perfect, we would conclude that the rank of $HC_{2N-n}(\xi) \oplus \dots \oplus HC_{2N+n}(\xi)$ is bigger than $r_B+2$, getting again a contradiction.

\end{document}